\documentclass[reqno]{amsart}
\usepackage{amsmath, amsthm, amscd, amsfonts, amssymb, graphicx, color, mathrsfs, stmaryrd}
\usepackage[bookmarksnumbered, colorlinks, plainpages]{hyperref}
\hypersetup{colorlinks=true,linkcolor=red, anchorcolor=green, citecolor=cyan, urlcolor=red, filecolor=magenta, pdftoolbar=true}

\usepackage[numbers]{natbib}

\newtheorem{thm}{Theorem}[section]
\newtheorem{lemma}[thm]{Lemma}
\newtheorem{pro}[thm]{Proposition}

\theoremstyle{definition}
\newtheorem{defn}[thm]{Definition}

\newtheorem{hyp}[thm]{Hypothesis}

\newtheorem*{ack}{Acknowledgements}
\theoremstyle{remark}

\numberwithin{equation}{section}

\newcommand{\ol}[1]{\overline{#1}}

\renewcommand{\hat}[1]{\widehat{#1}}
\renewcommand{\tilde}[1]{\widetilde{#1}}

\newcommand{\set}[1]{{\left\{#1\right\}}}
\newcommand{\pa}[1]{{\left(#1\right)}}
\newcommand{\sq}[1]{{\left[#1\right]}}
\newcommand{\gen}[1]{{\left\langle #1\right\rangle}}
\newcommand{\abs}[1]{{\left|#1\right|}}
\newcommand{\norm}[1]{{\left\|#1\right\|}}
\newcommand{\ssm}{\smallsetminus}

\newcommand{\ra}{\rightarrow}

\newcommand{\xra}{\xrightarrow}

\newcommand{\N}{\mathbb{N}}

\newcommand{\R}{\mathbb{R}}

\newcommand{\eqsys}[1]{{\left\{\begin{array}{ll}#1\end{array}\right.}}

\newcommand{\elle}{\operatorname{L}}

\newcommand{\tc}{\, \middle |\,}                                                % METTERE \middle

\newcommand{\con}{\operatorname{\mathscr{C}}}

\newcommand{\eps}{\varepsilon}

\newcommand{\fcon}{\operatorname{\mathscr{FC}}}

\DeclareMathOperator{\lip}{\operatorname{Lip}}

\DeclareMathOperator{\trace}{\operatorname{Tr}}

\DeclareMathOperator{\dom}{\operatorname{dom}}

\makeatletter
\def\thebibliography#1{\section*{References\@mkboth
{References}{References}}\list
{[\arabic{enumi}]}{\settowidth\labelwidth{[#1]}\leftmargin\labelwidth
\advance\leftmargin\labelsep
\usecounter{enumi}}
\def\newblock{\hskip .11em plus .33em minus .07em}
\sloppy\clubpenalty4000\widowpenalty4000 %\sfcode�\.=1000\relax
}
\makeatother

\begin{document}

\frenchspacing

\setcounter{page}{1}

\title[Sobolev Regularity for Weighted Elliptic Problems in convex subsets of Banach spaces]{Maximal Sobolev regularity for solutions of elliptic equations in Banach spaces endowed with a weighted Gaussian measure: the convex subset case}

\author[G. Cappa]{G. Cappa}

\address[G. Cappa]{Dipartimento di Matematica e Informatica, Universit\`a degli Studi di Parma, Parco Area delle Scienze 53/A, 43124 Parma, Italy.}
\email{\textcolor[rgb]{0.00,0.00,0.84}{gianluca.cappa@nemo.unipr.it}}

\author[S. Ferrari]{S. Ferrari$^*$}

\address[S. Ferrari]{Dipartimento di Matematica e Fisica ``Ennio de Giorgi'', Universit\`a del Salento, Via per Arnesano sn, 73100 Lecce, Italy.}
\email{\textcolor[rgb]{0.00,0.00,0.84}{simone.ferrari@unisalento.it}}

%\dedicatory{This paper is dedicated to Professor ABCD}



\date{\today\\ \indent $^*$ Corresponding author}

\begin{abstract}
Let $X$ be a separable Banach space endowed with a non-degenerate centered Gaussian measure $\mu$. The associated Cameron--Martin space is denoted by $H$. Consider two sufficiently regular convex functions $U:X\ra\R$ and $G:X\ra \R$. We let $\nu=e^{-U}\mu$ and $\Omega=G^{-1}(-\infty,0]$. In this paper we are interested in the $W^{2,2}$ regularity of the weak solutions of elliptic equations of the type
\begin{gather}\label{Probelma in abstract}
\lambda u-L_{\nu,\Omega} u=f,
\end{gather}
where $\lambda>0$, $f\in \elle^2(\Omega,\nu)$ and $L_{\nu,\Omega}$ is the self-adjoint operator associated with the quadratic form
\[(\psi,\varphi)\mapsto \int_\Omega\gen{\nabla_H\psi,\nabla_H\varphi}_Hd\nu\qquad\psi,\varphi\in W^{1,2}(\Omega,\nu).\]
In addition we will show that if $u$ is a weak solution of problem \eqref{Probelma in abstract} then it satisfies a Neumann type condition at the boundary, namely for $\rho$-a.e. $x\in G^{-1}(0)$
\[\gen{\trace(\nabla_Hu)(x),\trace(\nabla_H G)(x)}_H=0,\]
where $\rho$ is the Feyel--de La Pradelle Hausdorff--Gauss surface measure and $\trace$ is the trace operator.
\end{abstract}

\maketitle

%\modulolinenumbers[5]
%\linenumbers

\section{Introduction} \label{Introduction}

Let $X$ be a separable Banach space with norm $\norm{\cdot}_X$, endowed with a non-degenerate centered Gaussian measure $\mu$. The associated Cameron--Martin space is denoted by $H$, its inner product by $\gen{\cdot,\cdot}_H$ and its norm by $\abs{\cdot}_H$. The spaces $W^{1,p}(X,\mu)$ and $W^{2,p}(X,\mu)$ are the classical Sobolev spaces of the Malliavin calculus (see \cite{Bog98}).

In this paper we are interested in the study of maximal Sobolev regularity for the solution $u$ of the problem
\begin{gather}\label{Problema}
\lambda u(x)-L_{\nu,\Omega}u(x)=f(x)\qquad \mu\text{-a.e. }x\in \Omega,
\end{gather}
where $\lambda >0$, $\Omega$ is a convex subset of $X$, $\nu$ is a measure of the form $e^{-U}\mu$ with $U:X\ra\R$ a convex function, $f\in\elle^2(\Omega,\nu)$ and $L_{\nu,\Omega}$ is the operator associated to the quadratic form
\[(\psi,\varphi)\mapsto \int_\Omega\gen{\nabla_H\psi,\nabla_H\varphi}_Hd\nu\qquad\psi,\varphi\in W^{1,2}(\Omega,\nu),\]
where $\nabla_H \psi$ is the gradient along $H$ of $\psi$ and $W^{1,2}(\Omega,\nu)$ is the Sobolev space on $\Omega$ associated to the measure $\nu$ (see Section \ref{Notations and preliminaries}).

We need to clarify what we mean by \emph{solution} of problem \eqref{Problema}. We say that $u\in W^{1,2}(\Omega,\nu)$ is a \emph{weak solution} of problem \eqref{Problema} if
\[\lambda\int_\Omega u\varphi d\nu+\int_\Omega\gen{\nabla_H u,\nabla_H\varphi}_Hd\nu=\int_\Omega f\varphi d\nu\qquad\text{for every }\varphi\in W^{1,2}(\Omega,\nu).\]
Notice that if the weak solution $u$ of problem \eqref{Problema} is unique, then $u=R(\lambda,L_{\nu,\Omega})f$, the resolvent of $L_{\nu,\Omega}$.

Results about existence, uniqueness and regularity of the weak solutions of problem \eqref{Problema}, in domains with sufficiently regular boundary, are known in the finite dimensional case (see the classical books \cite{GT01} and \cite{LU68} for a bounded $\Omega$ and \cite{BF04}, \cite{DPL04}, \cite{LMP05}, \cite{DPL07} and \cite{DPL08} for an unbounded $\Omega$).
If $X$ is infinite dimensional separable Hilbert space then some maximal Sobolev regularity results are known. See for example \cite{BDPT09} and \cite{BDPT11} where $U\equiv 0$ and \cite{DPL15} where $U$ is bounded from below.

When $\Omega=X$ more results are known, see for example \cite{DPG01}, \cite{MPRS02} and \cite{BL07} if $X$ is finite dimensional, \cite{DPL14} if $X$ is a Hilbert space and \cite{CF16} if $X$ is a separable Banach space. If $X$ is general separable Banach space and $\Omega\varsubsetneq X$, then the only results about maximal Sobolev regularity are the ones in \cite{Cap16}, where problem \eqref{Problema} was studied when $U\equiv 0$, namely when $L_{\nu,\Omega}$ is the Ornstein--Uhlenbeck operator on $\Omega$. We do not know of any $W^{2,2}$ regularity results for solutions of problem \eqref{Problema} in subsets of infinite dimensional Banach spaces in the case $U\not\equiv 0$.

In order to state the main results of this paper we need some hypotheses on the set $\Omega$ and on the weighted measure $\nu$.

\begin{hyp}\label{ipotesi dominio}
Let $G:X\ra\R$ be a $(2,r)$-precise version (see Section \ref{Notations and preliminaries}), for some $r>1$, of a function belonging to $W^{2,q}(X,\mu)$ for every $q>1$ and assume
\begin{enumerate}
\item $\mu(G^{-1}(-\infty,0])>0$ and $G^{-1}(-\infty,0]$ is closed and convex;\label{ipo dominio convessita e chiusura}

\item $\abs{\nabla_H G}_H^{-1}\in\elle^q(G^{-1}(-\infty,0],\mu)$ for every $q>1$.\label{ipo dominio non degeneratezza}
\end{enumerate}
We set $\Omega:=G^{-1}(-\infty,0]$.
\end{hyp}
\noindent
We will recall the definition of $H$-closure in Section \ref{Notations and preliminaries}. All our results will be independent on our choice of a precise version of the function $G$ made in Hypothesis \ref{ipotesi dominio}. This hypothesis is taken from \cite{CL14} in order to define traces of Sobolev functions on level sets of $G$. The main differences between Hypothesis \ref{ipotesi dominio} and  the hypothesis contained in \cite{CL14} are the requirements of closure and convexity of the set $G^{-1}(-\infty,0]$. It will be clear when these two additional assumptions will be useful.

\begin{hyp}\label{ipotesi peso}
$U:X\ra\R\cup\set{+\infty}$ is a proper convex and lower semicontinuous function belonging to $W^{1,t}(X,\mu)$ for some $t>3$. We set $\nu:=e^{-U}\mu$.
\end{hyp}
\noindent
The assumption $t>3$ may sound strange, but it is needed to define the weighted Sobolev spaces $W^{1,2}(X,\nu)$. Indeed observe that, by \cite[Lemma 7.5]{AB06}, $e^{-U}$ belongs to $W^{1,r}(X,\mu)$ for every $r<t$. Thus if $U$ satisfies Hypothesis \ref{ipotesi peso}, then it satisfies \cite[Hypothesis 1.1]{Fer15}; namely $e^{-U}\in W^{1,s}(X,\mu)$ for some $s> 1$ and $U\in W^{1,r}(X,\mu)$ for some $r>s'$.
Then following \cite{Fer15} it is possible to define the space $W^{1,2}(X,\nu)$ as the domain of the closure of the gradient operator along $H$ (see Section \ref{Notations and preliminaries} for a in-depth discussion). We want to point out that Hypothesis \ref{ipotesi peso} is different from
\cite[Hypothesis 1.1]{CF16}, since here we require just lower semicontinuity instead of continuity. We simply realized that the continuity hypothesis was not needed. This implies that in the case $\Omega=X$ the results of this paper are also a generalization of the results in \cite{CF16}.

Our main result is a generalization of the main results of both \cite{DPL15} and \cite{Cap16}.

\begin{thm}\label{Main theorem 1}
Assume Hypotheses \ref{ipotesi dominio} and \ref{ipotesi peso} hold. For every $\lambda>0$ and $f\in\elle^2(\Omega,\nu)$ problem (\ref{Problema}) has a unique weak solution $u\in W^{2,2}(\Omega,\nu)$. In addition the following inequalities hold
\begin{gather*}
%\label{1 stime max}
\norm{u}_{\elle^2(\Omega,\nu)}\leq\frac{1}{\lambda}\norm{f}_{\elle^2(\Omega,\nu)};\qquad \norm{\nabla_H u}_{\elle^2(\Omega,\nu;H)}\leq\frac{1}{\sqrt{\lambda}}\norm{f}_{\elle^2(\Omega,\nu)};\\
%\label{2 stime max}
\|\nabla_H^2 u\|_{\elle^2(\Omega,\nu;\mathcal{H}_2)}\leq \sqrt{2}\norm{f}_{\elle^2(\Omega,\nu)},
\end{gather*}
where $\mathcal{H}_2$ is the space of Hilbert--Schmidt operators in $H$. See Definition \ref{definizione spazi di sobolev pesati} for the definition of the space $W^{2,2}(\Omega,\nu)$.
\end{thm}

The idea of the proof of Theorem \ref{Main theorem 1} is to approximate the solution of problem \eqref{Problema} by the solutions of penalized problems on the whole space. This method was already used in the papers \cite{BDPT09}, \cite{BDPT11} and \cite{DPL15}, where the authors used some properties of Hilbert spaces, namely the differentiability of the Moreau--Yosida approximations and of the square of the distance function. Due to the lack of differentiability, at least in general, of the natural norm of a separable Banach space these methods cannot be applied in our case. The idea behind this paper is to replace ``differentiability'' by ``differentiability along $H$'' which is sufficient for our goals, because we use a modification of the Moreau--Yosida approximations, which already appeared in \cite{CF16} (see Section \ref{Some properties of Moreau--Yoshida approximations along $H$}), and of the distance function (see Section \ref{Projection on convex set along $H$}).

The paper is organized in the following way: in Section \ref{Notations and preliminaries} we recall some basic definitions and we fix the notations. In Section \ref{Projection on convex set along $H$} we introduce the distance function $d_H(\cdot,\mathcal{C}):X\ra\R\cup\set{+\infty}$ defined as
\begin{gather*}
d_H(x,\mathcal{C})=\eqsys{\inf\set{\abs{h}_H\tc h\in H\cap (x-\mathcal{C})} & \text{ if }H\cap(x-\mathcal{C})\neq\emptyset;\\
+\infty & \text{ if }H\cap(x-\mathcal{C})=\emptyset,}
\end{gather*}
where $\mathcal{C}$ is a Borel subset of $X$. If $\mathcal{C}$ is a convex and closed set, we will prove some basic properties of such function that will be useful in the proof of Theorem \ref{Main theorem 1}. In Section \ref{Some properties of Moreau--Yoshida approximations along $H$} we recall the definition and some properties of the Moreau--Yosida approximations along $H$ (see \cite{CF16}). We will also prove an important property of the gradient along $H$ of the Moreau--Yosida approximations along $H$. Section \ref{Sobolev regularity estimates} is dedicated to the proof of Theorem \ref{Main theorem 1}, while in Section \ref{The Neumann condition} we will show that if $u\in W^{2,2}(\Omega,\nu)$ is a weak solution of equation \eqref{Problema}, then it satisfies a Neumann type condition at the boundary:

\begin{thm}\label{inclusione dominio}
Assume that Hypotheses \ref{ipotesi dominio} and \ref{ipotesi peso} hold. Let $\lambda>0$, $f\in\elle^2(\Omega,\nu)$ and let $u\in W^{2,2}(\Omega,\nu)$ be the weak solution of problem \eqref{Problema}, given by Theorem \ref{Main theorem 1}. Then for $\rho$-a.e. $x\in G^{-1}(0)$
\begin{gather}\label{Neumann equation}
\gen{\trace(\nabla_Hu)(x),\trace(\nabla_HG)(x)}_H=0,
\end{gather}
where $\rho$ is the Feyel--de La Pradelle Hausdorff--Gauss surface measure (see Section \ref{Notations and preliminaries} for the definition) and $\trace$ is the trace operator in the space $W^{1,2}(X,\nu;H)$ defined in Section \ref{Notations and preliminaries}.
\end{thm}
\noindent
%Equation \eqref{Neumann equation} can be seen as a Neumann type condition since it implies
%\[\gen{\trace(\nabla_Hu)(x),\trace\pa{\frac{\nabla_HG}{\abs{\nabla_H G}_H}}(x)}_H=0.\]
%So the vector field $\nabla_HG/\abs{\nabla G}_H$ play the role of the exterior unit normal vector to
%$G^{-1}(0)$. It is indeed the exterior unit normal vector if $G$ is smooth enough and $X$ is a separable
%Hilbert space.

Finally in Section \ref{Examples} we consider the Banach space $\con[0,1]$ of continuous functions on the closed interval $[0,1]$ with the sup norm, endowed with the classical Wiener measure (see \cite[Example 2.3.11 and Remark 2.3.13]{Bog98} for its construction). We study weights of the type
\[U_1(f)=\Phi\pa{\int_0^1f(\xi)d\tau(\xi)},\qquad \qquad U_2(f)=\int_0^1\Psi(f(\xi),\xi)d\xi,\]
where $\tau$ is a finite Borel measure in $[0,1]$, $f\in\con[0,1]$ and $\Phi:\R\ra\R$ and $\Psi:\R^2\ra\R$ are sufficiently regular convex functions. In addition $\Omega$ will be a closed halfspace or the set $\{f\in\con[0,1]\,|\, \int_0^1 f^2(\xi)d\xi\leq r\}$ for some $r>0$.

\section{Notations and preliminaries} \label{Notations and preliminaries}

We will denote by $X^*$ the topological dual of $X$. We recall that $X^*\subseteq\elle^2(X,\mu)$. The linear operator $R_\mu:X^*\ra (X^*)'$
\begin{gather}\label{operatore di covariaza}
R_\mu x^*(y^*)=\int_X x^*(x)y^*(x)d\mu(x)
\end{gather}
is called the covariance operator of $\mu$. Since $X$ is separable, then it is actually possible to prove that $R_\mu:X^*\ra X$ (see \cite[Theorem 3.2.3]{Bog98}). We denote by $X^*_\mu$ the closure of $X^*$ in $\elle^2(X,\mu)$. The covariance operator $R_\mu$ can be extended by continuity to the space $X^*_\mu$, still by formula \eqref{operatore di covariaza}. By \cite[Lemma 2.4.1]{Bog98} for every $h\in H$ there exists a unique $g\in X^*_\mu$ with $h= R_\mu g$, in this case we set
\begin{gather}\label{definizione hat}
\hat{h}:=g.
\end{gather}

Throughout the paper we fix an orthonormal basis $\set{e_i}_{i\in\N}$ of $H$ such that $\hat{e}_i$ belongs to $X^*$, for every $i\in\N$. Such basis exists by \cite[Corollary 3.2.8(ii)]{Bog98}.

\subsection{Differentiability along $H$}

We say that a function $f:X\ra\R$ is \emph{differentiable along $H$ at $x$} if there is $v\in H$ such that
\[\lim_{t\ra 0}\frac{f(x+th)-f(x)}{t}=\gen{v,h}_H\qquad\text{ uniformly for }h\in H\text{ with }\abs{h}_H=1.\]
In this case the vector $v\in H$ is unique and we set $\nabla_H f(x):=v$, moreover for every $k\in\N$ the derivative of $f$ in the direction of $e_k$ exists and it is given by
\begin{gather*}%\label{partial derivative definition}
\partial_k f(x):=\lim_{t\ra 0}\frac{f(x+te_k)-f(x)}{t}=\gen{\nabla_H f(x),e_k}_H.
\end{gather*}

We denote by $\mathcal{H}_2$ the space of the Hilbert--Schmidt operators in $H$, that is the space of the bounded linear operators $A:H\ra H$ such that $\norm{A}_{\mathcal{H}_2}^2=\sum_{i}\abs{Ae_i}^2_H$ is finite (see \cite{DU77}).
We say that a function $f:X\ra\R$ is \emph{two times differentiable along $H$ at $x$} if it is differentiable along $H$ at $x$ and $A\in\mathcal{H}_2$ exists such that
\[H\text{-}\lim_{t\ra 0}\frac{\nabla_Hf(x+th)-\nabla_Hf(x)}{t}=A h\qquad\text{ uniformly for }h\in H\text{ with }\abs{h}_H=1.\]
In this case the operator $A$ is unique and we set $\nabla_H^2 f(x):=A$. Moreover for every $i,j\in\N$ we set
\begin{gather*}%\label{partial derivative definition 2}
\partial_{ij} f(x):=\lim_{t\ra 0}\frac{\partial_jf(x+te_i)-\partial_jf(x)}{t}=\langle\nabla_H^2 f(x)e_j,e_i\rangle_H.
\end{gather*}

\subsection{Special classes of functions}

For $k\in\N\cup\set{\infty}$, we denote by $\fcon^k_b(X)$ the space of the cylindrical function of the type
\(f(x)=\varphi(x^*_1(x),\ldots,x^*_n(x))\)
where $\varphi\in\con^{k}_b(\R^n)$ and $x^*_1,\ldots,x^*_n\in X^*$ and $n\in\N$. We remark that $\fcon^\infty_b(X)$ is dense in $\elle^p(X,\nu)$ for all $p\geq 1$ (see \cite[Proposition 3.6]{Fer15}). We recall that if $f\in \fcon^2_b(X)$, then $\partial_{ij}f(x)=\partial_{ji}f(x)$ for every $i,j\in\N$ and $x\in X$.

If $Y$ is a Banach space, a function $F:X\ra Y$ is said to be $H$-Lipschitz if $C>0$ exists such that
\[\norm{F(x+h)-F(x)}_Y\leq C\abs{h}_H,\]
for every $h\in H$ and $\mu$-a.e. $x\in X$ (see \cite[Section 4.5 and Section 5.11]{Bog98}).

A function $F:X\ra\R$ is said to be $H$-continuous, if \(\lim_{H\ni h\ra 0}F(x+h)=F(x)\), for $\mu$-a.e. $x\in X$.

\subsection{Sobolev spaces}
The Gaussian Sobolev spaces $W^{1,p}(X,\mu)$ and $W^{2,p}(X,\mu)$, with $p\geq 1$, are the completions of the \emph{smooth cylindrical functions} $\fcon_b^\infty(X)$ in the norms
\begin{gather*}
\norm{f}_{W^{1,p}(X,\mu)}:=\norm{f}_{\elle^p(X,\mu)}+\pa{\int_X\abs{\nabla_H f(x)}_H^pd\mu(x)}^{\frac{1}{p}};\\
\norm{f}_{W^{2,p}(X,\mu)}:=\norm{f}_{W^{1,p}(X,\mu)}+\pa{\int_X\norm{\nabla_H^2 f(x)}^p_{\mathcal{H}_2}d\mu(x)}^{\frac{1}{p}}.
\end{gather*}
Such spaces can be identified with subspaces of $\elle^p(X,\mu)$ and the (generalized) gradient and Hessian along $H$, $\nabla_H f$ and $\nabla_H^2 f$, are well defined and belong to $\elle^p(X,\mu;H)$ and $\elle^p(X,\mu;\mathcal{H}_2)$, respectively. The spaces $W^{1,p}(X,\mu;H)$ are defined in a similar way, replacing smooth cylindrical functions with $H$-valued smooth cylindrical functions (i.e. the linear span of the functions $x\mapsto f(x)h$, where $f$ is a smooth cylindrical function and $h\in H$). For more information see \cite[Section 5.2]{Bog98}.

Now we consider $\nabla_H:\fcon^\infty_b(X)\ra\elle^p(X,\nu;H)$. This operator is closable in $\elle^p(X,\nu)$ whenever $p>\frac{t-1}{t-2}$ (see \cite[Definition 4.3]{Fer15}). For such $p$ we denote by $W^{1,p}(X,\nu)$ the domain of its closure in $\elle^p(X,\nu)$. In the same way the operator $(\nabla_H,\nabla^2_H):\fcon^\infty_b(X)\ra \elle^p(X,\nu;H)\times\elle^p(X,\nu;\mathcal{H}_2)$ is closable in $\elle^p(X,\nu)$, whenever $p>\frac{t-1}{t-2}$ (see \cite[Proposition 2.1]{CF16}). For such $p$ we denote by $W^{2,p}(X,\nu)$ the domain of its closure in $\elle^p(X,\nu)$. The spaces $W^{1,p}(X,\nu;H)$ are defined in a similar way, replacing smooth cylindrical functions with $H$-valued smooth cylindrical functions.

We want to point out that if Hypothesis \ref{ipotesi peso} holds, then $\frac{t-1}{t-2}<2$. In particular the above discussion allows us to define the Sobolev spaces $W^{1,2}(X,\nu)$ and $W^{2,2}(X,\nu)$.

We shall use the integration by parts formula (see \cite[Lemma 4.1]{Fer15}) for $\varphi\in W^{1,p}(X,\nu)$ with $p>\frac{t-1}{t-2}$:
\begin{gather*}%\label{int by part}
\int_X\partial_k\varphi d\nu=\int_X\varphi(\partial_kU+\hat{e}_k)d\nu\qquad\text{ for every }k\in\N,
\end{gather*}
where $\hat{e}_k$ is defined in formula \eqref{definizione hat}.

Throughout the paper we will use the following simplified version of \cite[Theorem 5.11.2]{Bog98} several times.

\begin{thm}\label{teorema 5.11.2}
Let $Y$ be either $\R$ or $H$, and let $F:X\ra Y$ be a measurable $H$-Lipschitz mapping. Then $F\in W^{1,p}(X,\mu;Y)$ for every $p>1$.
\end{thm}

\subsection{Capacities and versions}

Let $L_p$ be the infinitesimal generator of the \emph{Ornstein--Uhlenbeck semigroup} in $\elle^p(X,\mu)$
\[T(t)f(x):=\int_Xf\pa{e^{-t}x+\sqrt{1-e^{-2t}}y}d\mu(y)\qquad\text{ for }t>0.\]
The \emph{$C_{2,p}$-capacity} of an open set $A\subseteq X$ is
\[C_{2,p}(A):=\inf\set{\norm{f}_{\elle^p(X,\mu)}\tc (I-L_p)^{-1}f\geq 1\ \mu\text{-a.e. in }A}.\]
For a general Borel set $B\subseteq X$ we let $C_{2,p}(B)=\inf\set{C_{2,p}(A)\tc B\subseteq A\text{ open}}$. Let $f\in W^{2,p}(X,\mu)$, $f$ is an equivalence class of functions and we call every element ``version''. A version $\ol{f}$ of $f$ exists that is Borel measurable and \emph{$C_{2,p}$-quasicontinuous}, i.e. for every $\eps>0$ there exists an open set $A\subseteq X$ such that $C_{2,p}(A)\leq \eps$ and $\ol{f}_{|_{X\ssm A}}$ is continuous. See \cite[Theorem 5.9.6]{Bog98}. Such a version is called a \emph{$(2,p)$-precise version of $f$}. Two precise versions of the same $f$ agree outside sets with null $C_{2,p}$-capacity.

\subsection{Sobolev spaces on sublevel sets}

The proof of the results recalled in this subsection can be found in \cite{CL14} and \cite{Fer15}. Let $G$ be a function satisfying Hypothesis \ref{ipotesi dominio}. We are interested in Sobolev spaces on sublevel sets of $G$.

For $k\in \N\cup\set{\infty}$, we denote by $\fcon^k_b(\Omega)$ the space of the restriction to $\Omega$ of functions in $\fcon^k_b(X)$. The spaces $W^{1,p}(\Omega,\mu)$ and $W^{2,p}(\Omega,\mu)$ for $p\geq 1$ are defined as the domain of the closure of the operators $\nabla_H:\fcon_b^\infty(\Omega)\ra \elle^p(\Omega,\mu;H)$ and $(\nabla_H,\nabla_H^2):\fcon^\infty_b(\Omega)\ra\elle^p(\Omega,\mu;H)\times\elle^p(\Omega,\mu;\mathcal{H}_2)$. See \cite[Lemma 2.2]{CL14} and \cite[Proposition 1]{Cap16}.

We remind the reader that the operator $\nabla_H:\fcon^\infty_b(\Omega)\ra \elle^p(\Omega,\nu; H)$ is closable in $\elle^p(\Omega,\nu)$, whenever $p>\frac{t-1}{t-2}$ (see \cite[Proposition 6.1]{Fer15}). For such $p$ we denote by $W^{1,p}(\Omega,\nu)$ the domain of its closure in $\elle^p(\Omega,\nu)$ and we will still denote by $\nabla_H$ the closure operator.

In order to define the spaces $W^{2,p}(\Omega,\nu)$ we need the closability of the operator $(\nabla_H,\nabla_H^2)$ in $\elle^p(\Omega,\nu)$.
\begin{pro}\label{chiusura hessiano}
Assume Hypotheses \ref{ipotesi dominio} and \ref{ipotesi peso} hold. Then for every $p>\frac{t-1}{t-2}$, the operator $(\nabla_H,\nabla^2_H):\fcon_b^\infty(\Omega)\ra \elle^p(\Omega,\nu;H)\times\elle^p(\Omega,\nu;\mathcal{H}_2)$ is closable in $\elle^p(\Omega,\nu)$. The closure will be still denoted by $(\nabla_H,\nabla^2_H)$.
\end{pro}

\begin{proof}
Let $(f_k)_{k\in\N}\subseteq \fcon_b^\infty(\Omega)$ such that
\begin{gather*}
\begin{array}{cl}
\displaystyle\lim_{k\ra+\infty}f_k=0, &\text{ in }\elle^p(\Omega,\nu);\\
\displaystyle\lim_{k\ra+\infty}\nabla_H f_k=F, &\text{ in }\elle^p(\Omega,\nu;H);\\
\displaystyle\lim_{k\ra+\infty}\nabla_H^2 f_k=\Phi, &\text{ in }\elle^p(\Omega,\nu;\mathcal{H}_2).
\end{array}
\end{gather*}
We will assume that the functions $f_k$ are actually smooth cylindrical functions on the whole space $X$. By \cite[Proposition 6.1]{Fer15} we have that $F(x)=0$ for $\nu$-a.e. $x\in \Omega$. We want to prove that $\Phi(x)=0$ for $\nu$-a.e. $x\in \Omega$. Since the restrictions to $\Omega$ of the elements of $\fcon_b^\infty(X)$ are dense in $\elle^{p'}(\Omega,\nu)$ (as a consequence of the density of $\fcon_b^\infty(X)$ in $\elle^{p'}(X,\nu)$, see \cite[Proposition 3.6]{Fer15}), we just have to prove that
\[\int_{\Omega}\gen{\Phi(x)e_j,e_i}u(x)d\nu(x)=0\]
holds for every $i,j\in\N$ and $u\in\fcon^{\infty}_b(X)$.

Let $\eta:\R\ra\R$ a smooth function such that $\norm{\eta}_\infty\leq 1$, $\norm{\eta'}_\infty\leq 2$ and
\[\eta(\xi)=\eqsys{0 & \xi\geq -1\\
1 & \xi\leq -2}\]
Let $\eta_n(\xi):=\eta(n\xi)$ and $u_n(x)=u(x)\eta_n(G(x))$. Observe that $u_n$ converges pointwise $\nu$-a.e. to $u$ in $\Omega$ and $\abs{u_n}\leq \abs{u}$ $\nu$-a.e., then by Lebesgue's dominated convergence theorem
\[\lim_{n\ra+\infty}\int_{\Omega}\gen{\Phi(x)e_j,e_i}u_n(x)d\nu(x)=\int_{\Omega}\gen{\Phi(x)e_j,e_i}u(x)d\nu(x).\]
For every $n\in\N$ we have $u_n\in W^{1,r}(X,\nu)$ with every $r>1$, then
\[\partial_i u_n(x)=\partial_i u(x)\eta_n(G(x))+u(x)\eta_n'(G(x))\partial_i G(x);\]
see \cite[Proposition 4.5(5) and Proposition 4.6]{Fer15}. Observe that
\begin{gather*}
\int_X u_n\partial_{ij} f_kd\nu=\int_X \partial_j f_ku_n(\hat{e}_i+\partial_iU)d\nu-\int_X (\eta_n\circ G)\partial_j f_k \partial_i ud\nu-\int_X (\eta_n'\circ G)u\partial_j f_k \partial_i Gd\nu,
\end{gather*}
and the following estimate holds:
\begin{gather*}
\int_X\abs{\partial_{ij} f_k u_n-\gen{\Phi e_j,e_i}_H u}d\nu\leq \int_X\abs{\partial_{ij} f_k}\abs{u_n-u}d\nu+\int_X\abs{\partial_{ij} f_k-\gen{\Phi e_j,e_i}_H}\abs{u}d\nu\leq\\
\leq\pa{\int_X\abs{\partial_{ij} f_k}^pd\nu}^{\frac{1}{p}}\pa{\int_X\abs{u_n-u}^{p'}d\nu}^{\frac{1}{p'}}+\pa{\int_X\abs{\partial_{ij} f_k-\gen{\Phi e_j,e_i}_H}^p d\nu}^{\frac{1}{p}}\pa{\int_X\abs{u}^{p'}d\nu}^{\frac{1}{p'}},
\end{gather*}
this implies $\lim_{n\ra+\infty}\lim_{k\ra+\infty}\int_X u_n\partial_{ij} f_kd\nu=\int_{\Omega}\gen{\Phi e_j,e_i}u d\nu$. Furthermore for every $n\in\N$ we get
\begin{gather*}
\int_X\abs{(\eta_n\circ G)\partial_j f_k\partial_i u} d\nu\leq \int_X\abs{\partial_j f_k\partial_i u}d\nu\leq\pa{\int_X\abs{\partial_j f_k}^pd\nu}^{\frac{1}{p}}\pa{\int_X\abs{\partial_i u}^{p'}d\nu}^{\frac{1}{p'}}\xra{k\ra+\infty}0.\\
\intertext{Let $s>1$, then for every $n\in\N$}
\int_X\abs{(\eta'_n\circ G) u\partial_j f_k \partial_i G}d\nu\leq\\
\leq 2n\norm{u}_\infty\pa{\int_X\abs{\partial_j f_k}^p d\nu}^{\frac{1}{p}}\pa{\int_X\abs{\partial_i G}^{p's'} d\mu}^{\frac{1}{p's'}}\pa{\int_X e^{-sU}d\mu}^{\frac{1}{p's}}\xra{k\ra+\infty}0,\\
\intertext{where the last limit follows from Hypothesis \ref{ipotesi peso};}
\int_X\abs{\partial_j f_k u_n\hat{e}_i}d\nu\leq\norm{u}_\infty\pa{\int_X\abs{\partial_j f_k}^pd\nu}^{\frac{1}{p}}\pa{\int_X\abs{\hat{e}_i}^{p'}d\nu}^{\frac{1}{p'}}\xra{k\ra+\infty}0;\\
\intertext{Let $r>1$, then for every $n\in\N$ we have}
\int_X\abs{\partial_j f_k u_n \partial_i U}d\nu\leq\\
\leq \norm{u}_\infty\pa{\int_X\abs{\partial_j f_k}^pd\nu}^{\frac{1}{p}}\pa{\int_X e^{-rU}d\mu}^{\frac{1}{p'r}}\pa{\int_X\abs{\partial_i U}^{p'r'}d\mu}^{\frac{1}{p'r'}}\xra{k\ra+\infty}0,
\end{gather*}
and the last limit exists whenever $p'r'\leq t$.
\end{proof}
\noindent
We remark that in the proof of Proposition \ref{chiusura hessiano} we have not used the assumptions of convexity and closure of $\Omega$.

We are now able to define the Sobolev spaces $W^{2,p}(\Omega,\nu)$.

\begin{defn}\label{definizione spazi di sobolev pesati}
Assume that Hypotheses \ref{ipotesi dominio} and \ref{ipotesi peso} hold. For $p>\frac{t-1}{t-2}$ we denote by $W^{2,p}(\Omega,\nu)$ the domain of the closure of the operator $(\nabla_H,\nabla^2_H):\fcon_b^\infty(\Omega)\ra \elle^p(\Omega,\nu;H)\times\elle^p(\Omega,\nu;\mathcal{H}_2)$ in $\elle^p(\Omega,\nu)$.
\end{defn}
Finally we want to remark that if Hypotheses \ref{ipotesi dominio} and \ref{ipotesi peso} hold, then $\frac{t-1}{t-2}<2$. In particular the above discussion allows us to define the Sobolev spaces $W^{1,2}(\Omega,\nu)$ and $W^{2,2}(\Omega,\nu)$.

\subsection{Surface measures}
For a comprehensive treatment of surface measures in infinite dimensional Banach spaces with Gaussian measures we refer to \cite{FP91}, \cite{Fey01} and \cite{CL14}. We recall the definition of the Feyel--de La Pradelle Hausdorff--Gauss surface measure. If $m\geq 2$ and $F=\R^m$ equipped with a norm $\norm{\cdot}_F$, we define
\[d\theta^F(x)=\frac{1}{(2\pi)^{\frac{m}{2}}}e^{-\frac{\norm{x}_F^2}{2}}dH_{m-1}(x),\]
where $H_{m-1}$ is the spherical $(m-1)$-dimensional Hausdorff measure in $\R^m$, i.e.
\[H_{m-1}(A)=\lim_{\delta\ra 0}\inf\set{\sum_{n\in\N}w_{m-1}r_n^{m-1}\tc A\subseteq \bigcup_{n\in\N}B(x_n,r_n),\ r_n<\delta,\ \text{for every }n\in\N},\]
where $w_{m-1}=\pi^{\frac{m-1}{2}}(\Gamma(\frac{m+1}{2}))^{-1}$. For every $m$-dimensional $F\subseteq H$ we consider the orthogonal projection (along $H$) on $F$:
\[x\mapsto\sum_{n=1}^m\gen{x,f_n}_Hf_n\qquad x\in H\]
where $\set{f_n}_{n=1}^m$ is an orthonormal basis of $F$. There exists a $\mu$-measurable projection $\pi^F$ on $F$, defined in the whole $X$, that extends it (see \cite[Theorem 2.10.11]{Bog98}). We denote by $\tilde{F}:=\ker\pi^F$ and by $\mu_{\tilde{F}}$ the image of the measure $\mu$ on $\tilde{F}$ through $I-\pi^F$. Finally we denote by $\mu_F$ the image of the measure $\mu$ on $F$ through $\pi^F$, which is the standard Gaussian measure on $\R^m$ if we identify $F$ with $\R^m$. Let $A\subseteq X$ be a Borel set and identify $F$ with $\R^m$, we set
\[\rho^F(A):=\int_{\ker \pi^F}\theta^F(A_x)d\mu_{\tilde{F}}(x),\]
where $A_x=\set{y\in F\tc x+y\in A}$. The map $F\mapsto\rho^F(A)$ is well defined and increasing, namely if $F_1\subseteq F_2$ are finite dimensional subspaces of $H$, then $\rho^{F_1}(A)\leq \rho^{F_2}(A)$ (see \cite[Lemma 3.1]{AMP10} and \cite[Proposition 3.2]{Fey01}). The Feyel--de La Pradelle Hausdorff--Gauss surface measure is defined by
\[\rho(A)=\sup\set{\rho^F(A)\tc F\subseteq H,\ F\text{ is a finite dimensional subspace}}.\]

Finally we remind the reader of the following density result (see \cite[Proposition 7]{Cap16}).

\begin{pro}\label{densita nel bordo}
Assume Hypotheses \ref{ipotesi dominio} and \ref{ipotesi peso} hold and let $p>1$. If $g\in\elle^p(G^{-1}(0),\rho)$ is such that for every $\varphi\in\fcon^\infty_b(X)$
\[\int_{G^{-1}(0)}\varphi gd\rho=0\]
then $g(x)=0$ for $\rho$-a.e. $x\in G^{-1}(0)$.
\end{pro}

\subsection{Traces of Sobolev functions}
Traces of Sobolev functions in infinite dimensional Banach spaces are studied in \cite{CL14} in the Gaussian case and in \cite{Fer15} in the weighted Gaussian case. Assume that Hypotheses \ref{ipotesi dominio} and \ref{ipotesi peso} hold and let $p>\frac{t-1}{t-2}$. If $\varphi\in W^{1,p}(\Omega,\nu)$ we define the trace of $\varphi$ on $G^{-1}(0)$ as follows:
\[\trace\varphi=\lim_{n\ra+\infty}\varphi_{n_{|_{G^{-1}(0)}}}\qquad\text{in }\elle^{q}(G^{-1}(0),e^{-U}\rho)\text{ for every } q\in \sq{1,p\frac{t-2}{t-1}},\]
where $(\varphi_{n})_{n\in\N}$ is any sequence in $\lip_b(\Omega)$, the space of bounded and Lipschitz functions on $\Omega$, which converges in $W^{1,p}(\Omega,\nu)$ to $\varphi$. The definition does not depend on the choice of the sequence $(\varphi_n)_{n\in\N}$ in $\lip_b(\Omega)$ approximating $\varphi$ in $W^{1,p}(\Omega,\nu)$ (see \cite[Proposition 7.1]{Fer15}). In addition the following result holds.

\begin{pro}\label{trace continuity}
Assume Hypotheses \ref{ipotesi dominio} and \ref{ipotesi peso} hold. The operator \(\trace:W^{1,p}(\Omega,\nu)\ra\elle^q(G^{-1}(0),e^{-U}\rho)\) is continuous for every $p>\frac{t-1}{t-2}$ and $q\in \sq{1,p\frac{t-2}{t-1}}$. Moreover if $U\equiv 0$, then the trace operator is continuous from $W^{1,p}(\Omega,\mu)$ to $\elle^q(G^{-1}(0),\rho)$ for every $p>1$ and $q\in[1,p)$ (see \cite[Corollary 4.2]{CL14} and \cite[Corollary 7.3]{Fer15}).
\end{pro}

We will still denote by $\trace\Psi=\sum_{n=1}^{+\infty}(\trace\psi_n)e_n$ if $\Psi\in W^{1,p}(\Omega,\nu;H)$, for $p>\frac{t-1}{t-2}$, and $\psi_n=\gen{\Psi,e_n}_H$. The main result of \cite{Fer15} is the following integration by parts formula.

\begin{thm}\label{divergence theorem with traces}
Assume Hypotheses \ref{ipotesi dominio} and \ref{ipotesi peso} hold and let $p>\frac{t-1}{t-2}$. For every $\varphi\in W^{1,p}(\Omega,\nu)$ and $k\in\N$ we have
\[\int_{\Omega}\pa{\partial_k\varphi-\varphi\partial_k U-\varphi\hat{e}_k}d\nu=\int_{G^{-1}(0)}\trace(\varphi) \trace\pa{\frac{\partial_k G}{\abs{\nabla_H G}_H}}e^{-U}d\rho.\]
\end{thm}

%Using Theorem \ref{divergence theorem with traces} we can show that $L_{\nu,\Omega}$ has a nice
%expression when it acts on sufficiently regular functions, satisfying a Neumann condition at the
%border. Namely if $\varphi\in\con^1_b(X)$ with $\gen{\nabla_H\varphi(x),\nabla_H G(x)}_H=0$ for $\rho$-
%a.e $x\in G^{-1}(0)$, then
%\begin{gather*}
%L_{\nu,\Omega}\varphi=\sum_{i=1}^{+\infty}\partial_{ii}\varphi-\sum_{i=1}^{+\infty}(\partial_i U+\hat{e}
%_i)\partial_i\varphi
%\end{gather*}
%where the series converges in $\elle^2(\Omega,\nu)$. Observe that in the finite dimensional case the
%operator $L_{\nu,\Omega}$ is an elliptic operator with possibly unbounded coefficients.

\section{$H$-distance function}\label{Projection on convex set along $H$}

In this section we study some properties of the following function. Let $x\in X$ and let $\mathcal{C}\subseteq X$ be a Borel set. We define
\begin{gather}\label{definizione H distanza}
d_H(x,\mathcal{C}):=\eqsys{\inf\set{\abs{h}_H\tc h\in H\cap (x-\mathcal{C})} & \text{ if }H\cap(x-\mathcal{C})\neq\emptyset;\\
+\infty & \text{ if }H\cap(x-\mathcal{C})=\emptyset.}
\end{gather}
$d_H$ can be seen as a distance function from $\mathcal{C}$ along $H$. This function was already considered in \cite{Kus82}, \cite{UZ97}, \cite[Example 5.4.10]{Bog98} and \cite{Hin11}, but the results of this section are new. We remark that
\begin{gather*}
d_H(x,\mathcal{C})=\eqsys{\inf\set{\abs{x-w}_H\tc h\in \mathcal{C}\cap (x+H)} & \text{ if }\mathcal{C}\cap (x+H)\neq\emptyset;\\
+\infty & \text{ if }\mathcal{C}\cap (x+H)=\emptyset.}
\end{gather*}
which agrees with \cite[Definition 2.5]{Hin11}.

The aim of this section is to prove that the function $d_H^2(\cdot,\mathcal{C})$ is differentiable along $H$ $\mu$-a.e., whenever $\mathcal{C}$ is a closed and convex subset of $X$. The ideas of the proof are actually pretty similar to the classical arguments that can be found in \cite[Section 5.1]{Bre11}, but we need to pay special attention since $d_H$ is not globally defined and behaves well just along the directions of $H$.

For the rest of the paper we will denote by $D(\mathcal{C})$ the set
\[D(\mathcal{C}):=\set{x\in X\tc H\cap(x-\mathcal{C})\neq\emptyset},\]
whenever $\mathcal{C}\subseteq X$. We recall that $d_H(\cdot,\mathcal{C})$ is a measurable function (see \cite[Lemma 2.6]{Hin11}).

\begin{lemma}\label{lemma translation invaritant}
Let $\mathcal{C}$ be a Borel subset of $X$. Then \(D(\mathcal{C})\) is measurable and $H$-translation invariant, i.e. $D(\mathcal{C})+H=D(\mathcal{C})$. Moreover $\mathcal{C}\subseteq D(\mathcal{C})$ and if $\mu(\mathcal{C})>0$, then $\mu(D(\mathcal{C}))=1$
\end{lemma}

\begin{proof}
The measurability of $D(\mathcal{C})$ follows from the measurability of $d_H$. $D(\mathcal{C})\subseteq D(\mathcal{C})+H$ and $\mathcal{C}\subseteq D(C)$ are obvious. Let $x\in D(\mathcal{C})$ and $h\in H$, then by the very definition of $D(\mathcal{C})$, there exists $k\in H\cap(x-\mathcal{C})$. We have that $k+h\in H\cap(x+h-\mathcal{C})$. So $x+h\in D(\mathcal{C})$.
We recall that for $H$-translation invariant sets a zero-one law holds. Namely $\mu(D(\mathcal{C}))=0$ or $\mu(D(\mathcal{C}))=1$ (see \cite[Corollary 2.5.4]{Bog98} and \cite[Proposition 2.1]{Hin11}).
\end{proof}

\begin{pro}\label{esistenza punto di minimo distanza}
Let $\mathcal{C}\subseteq X$ be a closed convex set. For every $x\in D(\mathcal{C})$, there exists a unique $m(x,\mathcal{C})\in H \cap (x-\mathcal{C})$ such that
\begin{gather*}%\label{equazione esistenza minimo}
\abs{m(x,\mathcal{C})}_H=d_H(x,\mathcal{C}).
\end{gather*}
\end{pro}

\begin{proof}
By \cite[Proposition 11.14]{BC11} there exists $h\in H\cap (x-\mathcal{C})$ such that
\[\abs{h}_H=d_H(x,\mathcal{C}).\]
Assume $h_1,h_2\in H\cap (x-\mathcal{C})$ are such that $\abs{h_1}_H=\abs{h_2}_H=d_H(x,\mathcal{C})$. Observe that
\[\frac{h_1+h_2}{2}\in H\cap(x-\mathcal{C})\]
and $\abs{h_1}_H\leq |2^{-1}(h_1+h_2)|_H\leq 2^{-1}\abs{h_1}+2^{-1}\abs{h_2}_H=\abs{h_1}_H$. In particular $\abs{h_1}_H=\abs{h_2}_H=|2^{-1}(h_1+h_2)|_H$. So, by the strict convexity of $\abs{\cdot}_H$, we have $h_1=h_2$.
\end{proof}

\begin{pro}\label{caratterizzazione minimo}
Let $\mathcal{C}\subseteq X$ be a closed convex set and $x\in D(\mathcal{C})$. For $m\in H\cap(x-\mathcal{C})$, we have $m=m(x,\mathcal{C})$ if, and only if,
\begin{gather}\label{equazione minimo}
\gen{h-m,m}_H\geq 0,
\end{gather}
for every $h\in H\cap (x-\mathcal{C})$.
\end{pro}

\begin{proof}
Let $v_t=(1-t)h+tm(x,\mathcal{C})$ for $t\in(0,1)$ and $h\in H\cap (x-\mathcal{C})$. Since $v_t\in H\cap(x-\mathcal{C})$ we have $\abs{m(x,\mathcal{C})}_H\leq \abs{v_t}_H$. So
\[\abs{m(x,\mathcal{C})}_H^2\leq\abs{v_t}_H^2=(1-t)^2\abs{h}_H^2+2t(1-t)\gen{h,m(x,\mathcal{C})}_H+t^2\abs{m(x,\mathcal{C})}_H^2.\]
Dividing by $1-t$ we get
\[(1+t)\abs{m(x,\mathcal{C})}_H^2\leq (1-t)\abs{h}_H^2+2t\gen{h,m(x,\mathcal{C})}_H.\]
Letting $t\ra 1^-$ we obtain
\[\gen{h-m(x,\mathcal{C}),m(x,\mathcal{C})}_H\geq 0.\]
Now let $m\in H\cap(x-\mathcal{C})$ be an element satisfying inequality \eqref{equazione minimo}. For every $h\in H\cap(x-\mathcal{C})$ we have
\begin{gather*}
\abs{m}_H^2=\gen{m-h,m}_H+\gen{h,m}_H\leq \gen{h,m}_H \leq \abs{h}_H\abs{m}_H.
\end{gather*}
Thus we get $\abs{m}_H\leq\abs{h}_H$ for every $h\in H\cap(x-\mathcal{C})$.
\end{proof}

\begin{pro}\label{punto minimo lip}
Let $\mathcal{C}\subseteq X$ be a closed convex set and $x\in D(\mathcal{C})$. The map $m_{x,\mathcal{C}}:H\ra H$ defined as $m_{x,\mathcal{C}}(h):=m(x+h,\mathcal{C})$ is well defined for every $h\in H$ and Lipschitz continuous, with Lipschitz constant less or equal than $1$.
\end{pro}

\begin{proof}
If $x\in D(\mathcal{C})$, then by Lemma \ref{lemma translation invaritant} $x+h\in D(\mathcal{C})$ for every $h\in H$. So by Proposition \ref{esistenza punto di minimo distanza} the element $m(x+h,\mathcal{C}) \in H\cap(x+h-\mathcal{C})$ such that $d_H(x+h,\mathcal{C})=\abs{m(x+h,\mathcal{C})}_H$ exists and it is unique. Thus the map $m_{x,\mathcal{C}}(h):=m(x+h,\mathcal{C})$ is well defined for every $h\in H$.

By the very definition of the function $m(\cdot,\mathcal{C})$ (see Proposition \ref{esistenza punto di minimo distanza}) we get that for every $x\in X$ and $h\in H$, $m(x,\mathcal{C})+h\in H\cap(x+h-\mathcal{C})$ and $m(x+h,\mathcal{C})-h\in H\cap(x-\mathcal{C})$. By Proposition \ref{caratterizzazione minimo} we get
\begin{gather}
\label{dis1} 0\leq\gen{k-m(x,\mathcal{C}),m(x,\mathcal{C})}_H\qquad\text{ for every }k\in H\cap(x-\mathcal{C});\\
\label{dis2} 0\leq\gen{l-m(x+h,\mathcal{C}),m(x+h,\mathcal{C})}_H\qquad\text{ for every }l\in H\cap(x+h-\mathcal{C}).
\end{gather}
Set $k=m(x+h,\mathcal{C})-h$ and $l=m(x,\mathcal{C})+h$ and sum inequalities \eqref{dis1} and \eqref{dis2}
\begin{gather}\label{dis che serve una volta sola}
\begin{array}{c}
\displaystyle 0\leq\gen{m(x+h,\mathcal{C})-h-m(x,\mathcal{C}),m(x,\mathcal{C})}_H+\gen{m(x,\mathcal{C})+h-m(x+h,\mathcal{C}),m(x+h,\mathcal{C})}_H=\\
\displaystyle= -\abs{m(x+h,\mathcal{C})-m(x,\mathcal{C})}_H^2+\gen{h,m(x+h,\mathcal{C})-m(x,\mathcal{C})}_H.
\end{array}
\end{gather}
By the Cauchy--Schwarz inequality, we get $\abs{m(x+h,\mathcal{C})-m(x,\mathcal{C})}_H\leq\abs{h}_H$.
\end{proof}

We remark that by inequality \eqref{dis che serve una volta sola} we get
\begin{gather}\label{dis che serve una sola volta 2}
\gen{m(x+h,\mathcal{C})-m(x,\mathcal{C}),h}_H\geq 0
\end{gather}
for every $x\in D(\mathcal{C})$ and $h\in H$. This fact will come in handy in the next proposition.

\begin{pro}\label{differenziabilita distanza}
Let $\mathcal{C}\subseteq X$ be a closed convex set. The function $d^2_H(\cdot,\mathcal{C})$ is differentiable along $H$ at every point $x\in D(\mathcal{C})$. Furthermore
\[(\nabla_H d_H^2(\cdot,\mathcal{C}))(x)=2m(x,\mathcal{C}).\]
\end{pro}

\begin{proof}
By the very definition of the function $m(\cdot,\mathcal{C})$ (see Proposition \ref{esistenza punto di minimo distanza}) we get that for every $x\in X$ and $h\in H$, $m(x,\mathcal{C})+h\in H\cap(x+h-\mathcal{C})$ and $m(x+h,\mathcal{C})-h\in H\cap(x-\mathcal{C})$. So by definition in formula \eqref{definizione H distanza} and Proposition \ref{esistenza punto di minimo distanza} we have
\begin{gather}\label{disuguaglianze in diff distanza}
\abs{m(x+h,\mathcal{C})}_H\leq\abs{m(x,\mathcal{C})+h}_H\quad\text{ and }\quad\abs{m(x,\mathcal{C})}_H\leq\abs{m(x+h,\mathcal{C})-h}_H.
\end{gather}
Now using the left hand side inequality in \eqref{disuguaglianze in diff distanza} we get
\begin{gather}\label{diff dist 1}
\begin{array}{c}
\displaystyle d^2_H(x+h,\mathcal{C})-d_H^2(x,\mathcal{C})-2\gen{m(x,\mathcal{C}),h}_H=\\
\displaystyle=\abs{m(x+h,\mathcal{C})}_H^2-\abs{m(x,\mathcal{C})}_H^2-2\gen{m(x,\mathcal{C}),h}_H\leq\\
\displaystyle\leq \abs{m(x,\mathcal{C})+h}_H^2-\abs{m(x,\mathcal{C})}_H^2-2\gen{m(x,\mathcal{C}),h}_H=\\
\displaystyle= \abs{m(x,\mathcal{C})}_H^2+\abs{h}_H^2+2\gen{m(x,\mathcal{C}),h}_H-\abs{m(x,\mathcal{C})}_H^2-2\gen{m(x,\mathcal{C}),h}_H=\abs{h}_H^2
\end{array}
\end{gather}
In addition using the right hand side inequality in  \eqref{disuguaglianze in diff distanza} and inequality \eqref{dis che serve una sola volta 2} we get
\begin{gather}\label{diff dist 2}
\begin{array}{c}
\displaystyle d^2_H(x+h,\mathcal{C})-d_H^2(x,\mathcal{C})-2\gen{m(x,\mathcal{C}),h}_H=\\
\displaystyle=\abs{m(x+h,\mathcal{C})}_H^2-\abs{m(x,\mathcal{C})}_H^2-2\gen{m(x,\mathcal{C}),h}_H\geq\\
\displaystyle\geq\abs{m(x+h,\mathcal{C})}_H^2-\abs{m(x+h,\mathcal{C})-h}_H^2-2\gen{m(x,\mathcal{C}),h}_H=\\
\displaystyle= \abs{m(x+h,\mathcal{C})}_H^2-\abs{m(x+h,\mathcal{C})}_H^2-\abs{h}_H^2+2\gen{m(x+h,\mathcal{C})-m(x,\mathcal{C}),h}_H\geq -\abs{h}_H^2
\end{array}
\end{gather}
Combining inequalities \eqref{diff dist 1} and \eqref{diff dist 2} we get
\[-\abs{h}_H\leq \frac{d^2_H(x+h,\mathcal{C})-d_H^2(x,\mathcal{C})-2\gen{m(x,\mathcal{C}),h}_H}{\abs{h}_H}\leq\abs{h}_H.\]
Letting $\abs{h}_H\ra 0$ we get the assertions of our proposition.
\end{proof}

\begin{pro}\label{convessita e continuita distanza}
Let $\mathcal{C}\subseteq X$ be a closed convex set such that $\mu(D(\mathcal{C}))=1$. The function $d_H^2(\cdot,\mathcal{C})$ is convex, $H$-continuous and belongs to $W^{1,p}(X,\mu)$ for every $p>1$.
\end{pro}

\begin{proof}
Convexity follows from a standard argument: let $\eps>0$, $x,y\in D(\mathcal{C})$ and choose $h_\eps(x)\in H\cap (x-\mathcal{C})$ and $h_\eps(y)\in H\cap (y-\mathcal{C})$ such that
\[\abs{h_\eps(x)}_H^2\leq d_H^2(x,\mathcal{C})+\eps\qquad \text{ and }\qquad\abs{h_\eps(y)}_H^2\leq d_H^2(y,\mathcal{C})+\eps.\]
Observe that $\lambda h_\eps(x)+(1-\lambda)h_\eps(y)\in H\cap(\lambda x+(1-\lambda)y-\mathcal{C})$, then
\begin{gather*}
d_H^2(\lambda x+(1-\lambda)y,\mathcal{C})\leq \abs{\lambda h_\eps(x)+(1-\lambda)h_\eps(y)}_H^2\leq \lambda\abs{h_\eps(x)}^2_H+(1-\lambda)\abs{h_\eps(y)}_H^2\leq\\
\leq \lambda d_H^2(x,\mathcal{C})+(1-\lambda)d_H^2(y,\mathcal{C})+\eps.
\end{gather*}
Letting $\eps\ra 0$ we get the convexity of $d_H^2(\cdot,\mathcal{C})$. If neither $x\in D(\mathcal{C})$ nor $y\in D(\mathcal{C})$, then for every $\lambda\in [0,1]$ we have
\[d_H^2(\lambda x+(1-\lambda)y,\mathcal{C})\leq \lambda d_H^2(x,\mathcal{C})+(1-\lambda)d_H^2(y,\mathcal{C})=+\infty.\]

$H$-continuity follows from Proposition \ref{esistenza punto di minimo distanza} and Proposition \ref{punto minimo lip}, since for every $x\in D(\mathcal{C})$ and $k\in H$
\begin{gather*}
\abs{d_H(x+k,\mathcal{C})-d_H(x,\mathcal{C})}=\abs{\abs{m(x+k,\mathcal{C})}_H-\abs{m(x,\mathcal{C})}_H}\leq\\
\leq\abs{m(x+k,\mathcal{C})-m(x,\mathcal{C})}_H\leq\abs{k}_H.
\end{gather*}
So $d_H^2(\cdot,\mathcal{C})$ is the composition of a $H$-Lipschitz function and a continuous function, then it is $H$-continuous.

The functions $d_H(\cdot,\mathcal{C})$ and $m(\cdot,\mathcal{C})$ are $H$-Lipschitz. Since $d_H^2(\cdot,\mathcal{C})$ is the composition of a smooth function and a $H$-Lipschitz function and it has $H$-Lipschitz gradient, we get $d_H^2(\cdot,\mathcal{C})\in W^{1,p}(X,\mu)$ for every $p>1$.
\end{proof}

\begin{pro}\label{caratterizzation distanza zero}
Let $\mathcal{C}\subseteq X$ be a closed convex set. Then $d_H(x,\mathcal{C})=0$ if, and only if, $x\in \mathcal{C}$.
\end{pro}

\begin{proof}
We remark that if $x\in \mathcal{C}$, then $0\in H\cap (x-\mathcal{C})$ and $d_H(x,\mathcal{C})=0$.

We recall that $H$ is continuously embedded in $X$ (see \cite[Theorem 2.4.5]{Bog98}), so there exists $K>0$ such that $\norm{h}_X\leq K\abs{h}_H$, for every $h\in H$. Let $x\in X$ such that $d_H(x,\mathcal{C})=0$. By definition we have $x\in D(\mathcal{C})$. Then a sequence $\set{h_n}_{n\in\N}\subseteq H\cap(x-\mathcal{C})$ exists such that $\abs{h_n}_H\ra 0$. For every $n\in\N$ there exists $c_n\in \mathcal{C}$ such that $h_n=x-c_n$, so
\[\lim_{n\ra+\infty}\norm{x-c_n}_X=\lim_{n\ra+\infty}\norm{h_n}_X\leq K\lim_{n\ra+\infty}\abs{h_n}_H= 0.\]
Thus, by the closure of $\mathcal{C}$, $x\in \mathcal{C}$.
\end{proof}

\section{A property of the Moreau--Yosida approximations along $H$}\label{Some properties of Moreau--Yoshida approximations along $H$}

We start this section by recalling the definition and some basic properties of the subdifferential of a convex continuous function. If $f:X\ra\R$ is a proper, convex and lower semicontinuous function, we denote by $\dom(f)$ the domain of $f$, namely $\dom(f):=\set{x\in X\tc f(x)<+\infty}$,
and by $\partial f(x)$ the subdifferential of $f$ at the point $x$, i.e.
\begin{gather}\label{subdifferenziale}
\partial f(x):=\left\{\begin{array}{lr}
\set{x^*\in X^*\tc f(y)\geq f(x)+x^*(y-x)\text{ for every }y\in X} & x\in\dom(f);\\
\emptyset & x\notin\dom(f).
\end{array}\right.
\end{gather}
We recall that $\partial f(x)$ is convex and weak-star compact for every $x\in X$, but it may be empty even if $x\in\dom(f)$. Furthermore $\partial f$ is a monotone operator, namely for every $x,y\in X$, $x^*\in\partial f(x)$ and $y^*\in \partial f(y)$ the following inequality holds
\begin{gather}\label{monotonicita subdifferenziale}
y^*(y-x)\geq x^*(y-x).
\end{gather}
For a classical treatment of monotone operators and subdifferential of convex functions we refer to \cite{Phe93} and \cite{BP12}.

We recall that for $\alpha>0$ the \emph{Moreau--Yosida approximation along $H$} of a proper convex and lower semicontinuous function $f:X\ra\R\cup\set{+\infty}$ is
\begin{gather}\label{M env}
f_\alpha(x):=\inf\set{f(x+h)+\frac{1}{2\alpha}\abs{h}^2_H\tc h\in H}.
\end{gather}
See \cite[Section 3]{CF16} for more details and \cite{Bre73} and \cite[Section 12.4]{BC11} for a treatment of the classical Moreau--Yosida approximations in Hilbert spaces, which are different from the one defined in \eqref{M env}. In the following proposition we recall some results contained in \cite[Section 3]{CF16}.

\begin{pro}\label{proprieta MY}
Let $x\in $, $\alpha>0$ and $f:X\ra\R\cup\set{+\infty}$ be a proper convex and lower semicontinuous function. The following properties hold:
\begin{enumerate}
\item the function $g_{\alpha,x}:H\ra\R$ defined as \(g_{\alpha,x}(h):=f(x+h)+\frac{1}{2\alpha}\abs{h}^2_H\), has a unique global minimum point $P(x,\alpha)\in H$. Moreover $P(x,\alpha)\ra 0$ in $H$ as $\alpha$ goes to zero;\label{esistenza minimo}

\item $f_\alpha(x)\nearrow f(x)$ as $\alpha\ra 0^+$. In particular $f_\alpha(x)\leq f(x)$ for every $\alpha>0$ and $x\in X$;\label{convergenza MY}

\item for $p\in H$, we have $p=P(x,\alpha)$ if, and only if, \(f(x+p)\leq f(x+h)+\frac{1}{\alpha}\gen{p,h-p}_H\), for every $h\in H$;\label{caratterizzazione punto minimo MY}

\item the function $P_{x,\alpha}:H\ra H$ defined as $P_{x,\alpha}(h):=P(x+h,\alpha)$ is Lipschitz continuous, with Lipschitz constant less or equal than $1$;

\item $f_\alpha$ is differentiable along $H$ at every point $x\in X$. In addition, for every $x\in X$, we have \(\nabla_H f_\alpha(x)=-\alpha^{-1}P(x,\alpha)\).\label{differenziabilita MY}

\item $f_\alpha$ belongs to $W^{2,p}(X,\mu)$, whenever $f\in\elle^p(X,\mu)$ for some $1\leq p<+\infty$.\label{Sobolev MY}
\end{enumerate}
\end{pro}

We will dedicate this section to prove that for every $x\in X$, $\nabla_H f_\alpha(x)$ converges to $\nabla_H f(x)$ as $\alpha$ goes to zero. In order to obtain such result we need a couple of lemmata.

\begin{lemma}\label{lemma inclusioni subdifferenziale}
Let $f:X\ra\R\cup\set{+\infty}$ be a proper convex and lower semicontinuous function, belonging to $W^{1,p}(X,\mu)$ for some $p>1$. Let $x\in\dom(f)$ and $\alpha>0$. If we let $F:H\ra \R$ be the function defined as $F(h):=f(x+h)$, then $F$ is proper convex and lower semicontinuous. Moreover $\nabla_H f(x)\in\partial F(0)$ and $\nabla_H f_\alpha(x)\in \partial F(P(x,\alpha))$.
\end{lemma}

\begin{proof}
Convexity and properness are trivial. Let $H$-$\lim_{n\ra+\infty}h_n=h$. Since $H$ is continuously embedded in $X$, $X$-$\lim_{n\ra+\infty} h_n=h$. By the fact that $f$ is $\norm{\cdot}_X$-lower semicontinuous, we get
\[F(h)=f(x+h)\leq\liminf_{n\ra+\infty}f(x+h_n)=\liminf_{n\ra+\infty}F(h_h).\]
So $F$ is $\abs{\cdot}_H$-lower semicontinuous. Since $x\in\dom(f)$, then $0\in \dom(F)$. In addition, by Proposition \ref{proprieta MY}\eqref{esistenza minimo} we have $f_\alpha(x)=f(x+P(x,\alpha))+(2\alpha)^{-1}\abs{P(x,\alpha)}_H^2\leq f(x)$, so
\[F(P(x,\alpha))=f(x+P(x,\alpha))\leq f(x)-\frac{1}{2\alpha}\abs{P(x,\alpha)}_H^2<+\infty.\]
This implies $P(x,\alpha)\in \dom(F)$. Let $h\in H$, then by \cite[Proof of proposition 3.1]{UZ96} we get
\[F(h)=f(x+h)\geq f(x)+\gen{\nabla_H f(x), h}_H=F(0)+\gen{\nabla_H f(x),h}_H.\]
Moreover by Proposition \ref{proprieta MY}\eqref{caratterizzazione punto minimo MY} and Proposition \ref{proprieta MY}\eqref{differenziabilita MY} we get
\begin{gather*}
F(h)=f(x+h)\geq f(x+P(x,\alpha))-\frac{1}{\alpha}\gen{P(x,\alpha),h-P(x,\alpha)}_H=\\
=F(P(x,\alpha))+\gen{\nabla_H f_\alpha(x),h-P(x,\alpha)}_H.
\end{gather*}
By formula \eqref{subdifferenziale}, we get $\nabla_H f(x)\in\partial F(0)$ and $\nabla_H f_\alpha(x)\in \partial F(P(x,\alpha))$.
\end{proof}

\begin{lemma}\label{MY of MY}
Let $f:X\ra\R\cup\set{+\infty}$ be a proper convex and lower semicontinuous function, belonging to $W^{1,p}(X,\mu)$ for some $p>1$. Let $x\in\dom(f)$ and $\alpha,\beta>0$. Then
\begin{gather}\label{MY alpha+beta}
(f_\alpha)_\beta(x)=f_{\alpha+\beta}(x).
\end{gather}
In particular $\nabla_H (f_\alpha)_\beta(x)=\nabla_H f_{\alpha+\beta}(x)$ and
\begin{gather}\label{monoton MY}
\abs{\nabla_H f_{\alpha+\beta}(x)}_H\leq \abs{\nabla_H f_\alpha(x)}_H,
\end{gather}
for every $x\in X$.
\end{lemma}

\begin{proof}
The proof of equality \eqref{MY alpha+beta} is similar to the one in \cite[Proposition 12.22]{BC11}, we give it just for the sake of completeness.
\begin{gather*}
(f_\alpha)_\beta(x)=\inf_{h\in H}\set{f_\alpha(x+h)+\frac{1}{2\beta}\abs{h}_H^2}=\\
=\inf_{h\in H}\set{\inf_{k\in H}\set{f(x+h+k)+\frac{1}{2\alpha}\abs{k}_H^2}+\frac{1}{2\beta}\abs{h}_H^2}=\\
=\inf_{h\in H}\set{\inf_{w\in H}\set{f(x+w)+\frac{1}{2\alpha}\abs{w-h}_H^2}+\frac{1}{2\beta}\abs{h}_H^2}=
\end{gather*}
\begin{gather*}
=\inf_{w\in H}\set{f(x+w)+\inf_{h\in H}\set{\frac{1}{2\alpha}\abs{w-h}_H^2+\frac{1}{2\beta}\abs{h}_H^2}}=
\end{gather*}
\begin{gather*}
=\inf_{w\in H}\set{f(x+w)+\frac{\alpha+\beta}{2\alpha\beta}\inf_{h\in H}\set{\frac{\beta}{\alpha+\beta}\abs{w-h}_H^2+\frac{\alpha}{\alpha+\beta}\abs{h}_H^2}}=
\end{gather*}
\begin{gather*}
=\inf_{w\in H}\set{f(x+w)+\frac{\alpha+\beta}{2\alpha\beta}\inf_{h\in H}\set{\frac{\alpha\beta}{(\alpha+\beta)^2}\abs{w}^2_H+\abs{\frac{\beta}{\alpha+\beta}w-h}_H^2}}=\\
=\inf_{w\in H}\set{f(x+w)+\frac{1}{2(\alpha+\beta)}\abs{w}^2_H}=f_{\alpha+\beta}(x).
\end{gather*}
So $(f_\alpha)_\beta(x)=f_{\alpha+\beta}(x)$.

We will now prove inequality \eqref{monoton MY}. Let $x\in X$ and $\alpha,\beta>0$. By Lemma \ref{lemma inclusioni subdifferenziale} we get $\nabla_H f_\alpha(x)\in \partial F_\alpha(0)$ and $\nabla_H(f_\alpha)_\beta(x)\in\partial F_\alpha(P_\alpha(x,\beta))$, where $F_\alpha(h):=f_\alpha(x+h)$ for $h\in H$ and $P_\alpha(x,\beta)$ is the unique minimum point of the function \(f_\alpha(x+h)+\frac{1}{2\beta}\abs{h}_H^2\). Such minimum exists by Proposition \ref{proprieta MY}\eqref{esistenza minimo}. By the monotonicity of the subdifferential (formula \eqref{monotonicita subdifferenziale}), Proposition \ref{proprieta MY}\eqref{differenziabilita MY} and equality \eqref{MY alpha+beta} we get
\begin{gather*}
0\leq\gen{\nabla_H f_\alpha(x)-\nabla_H(f_\alpha)_\beta(x),-P_\alpha(x,\beta)}_H\leq \beta\gen{\nabla_H f_\alpha(x)-\nabla_H(f_\alpha)_\beta(x),-\beta^{-1}P_\alpha(x,\beta)}_H\leq\\
\leq \beta\gen{\nabla_H f_\alpha(x)-\nabla_H(f_\alpha)_\beta(x),\nabla_H(f_\alpha)_\beta(x)}_H= \beta\gen{\nabla_H f_\alpha(x)-\nabla_H f_{\alpha+\beta}(x),\nabla_H f_{\alpha+\beta}(x)}_H.
\end{gather*}
So $\abs{\nabla_H f_{\alpha+\beta}(x)}_H^2\leq \gen{\nabla_H f_\alpha(x),\nabla_H f_{\alpha+\beta}(x)}_H$, and $\abs{\nabla_H f_{\alpha+\beta}(x)}_H\leq\abs{\nabla_H f_\alpha(x)}_H$.
\end{proof}

Now we have all the ingredients required to prove a convergence results about $\nabla_H f_\alpha$.

\begin{pro}\label{convergenza gradiente MY}
Let $f:X\ra\R\cup\set{+\infty}$ be a proper convex and lower semicontinuous function, belonging to $W^{1,p}(X,\mu)$ for some $p>1$. Let $x\in\dom(f)$. Then $\abs{\nabla_H f_\alpha(x)}_H\nearrow\abs{\nabla_H f(x)}_H$ as $\alpha\ra 0^+$ for $\mu$-a.e. $x\in X$. In particular
\begin{gather}\label{monotonia convergenza gradiente MY}
\abs{\nabla_H f_\alpha(x)}_H\leq \abs{\nabla_H f(x)}_H
\end{gather}
for $\mu$-a.e. $x\in X$ and for every $\alpha>0$.
\end{pro}

\begin{proof}
Let $x\in \dom(f)$ and $\alpha>0$. By Lemma \ref{lemma inclusioni subdifferenziale} we get $\nabla_H f_\alpha(x)\in \partial F(P(x,\alpha))$, where $F(h):=f(x+h)$ for $h\in H$ and $P(x,\alpha)$ is the unique minimum point of the function \(f(x+h)+\frac{1}{2\alpha}\abs{h}_H^2\). Such minimum exists by Proposition \ref{proprieta MY}\eqref{esistenza minimo}. By the weak compactness of the subdifferential there exists a point of minimal norm $h_0\in H$ in $\partial F(0)$.

By the monotonicity of the subdifferential (formula \eqref{monotonicita subdifferenziale}) we have
\begin{gather*}
0\leq \gen{P(x,\alpha),\nabla_H f_\alpha(x)-h_0}_H=\alpha\gen{-\alpha^{-1}P(x,\alpha),h_0-\nabla_H f_\alpha(x)}_H
=\alpha\gen{\nabla_H f_\alpha(x),h_0-\nabla_H f_\alpha(x)}_H,
\end{gather*}
where the last equality follows from Lemma \ref{proprieta MY}\eqref{differenziabilita MY}. By the Cauchy--Schwarz inequality we get
\begin{gather}\label{disuguaglianza 2}
\abs{\nabla_H f_{\alpha}(x)}_H^2\leq \gen{h_0,\nabla_H f_{\alpha}(x)}_H.
\end{gather}
Using inequality \eqref{disuguaglianza 2} we get
\begin{gather}
\label{disuguaglianza 3}\displaystyle \abs{\nabla_Hf_\alpha}_H\leq\abs{h_0}_H;\\
\label{disuguaglianza 3.1}\displaystyle\abs{\nabla_H f_\alpha (x)-h_0}_H^2=\abs{\nabla_H f_\alpha(x)}^2_H-2\gen{h_0,\nabla_Hf_\alpha(x)}_H+\abs{h_0}_H^2\leq \abs{h_0}_H^2-\abs{\nabla_Hf_\alpha(x)}^2_H.
\end{gather}

By inequality \eqref{disuguaglianza 3} we get that the set $\set{\nabla_H f_\alpha(x)\tc \alpha>0} $ is bounded in $H$. Let $(\alpha_n)_{n\in\N}$ be a sequence converging to zero. By weak compactness a subsequence, that we will still denote by $(\alpha_n)_{n\in\N}$, and $y\in H$ exist such that $\nabla_H f_{\alpha_n}(x)$ weakly converges to $y$ as $n$ goes to $+\infty$. By inequality \eqref{disuguaglianza 3} and weakly lower semicontinuity of $\abs{\cdot}_H$ we have that
\begin{gather}\label{dis lsc}
\abs{y}_H\leq \lim_{n\ra +\infty}\abs{\nabla_H f_{\alpha_n}(x)}_H\leq \abs{h_0}_H.
\end{gather}
We claim that $y\in \partial F(0)$, indeed recalling that $\abs{P(x,\alpha_n)}_H\ra 0$ as $n$ goes to $+\infty$ (Proposition \ref{proprieta MY}\eqref{esistenza minimo}), $\set{\nabla_H f_\alpha(x)\tc \alpha>0} $ is bounded in $H$ and $f$ is lower semicontinuous we have
\begin{gather*}
\gen{y,h}_H=\lim_{n\ra+\infty}\gen{\nabla_H f_{\alpha_n}(x),h}_H\leq\lim_{n\ra+\infty}\pa{f(x+h)-f(x+P(x,\alpha_n))+\gen{\nabla_H f_{\alpha_n}(x),P(x,\alpha_n)}_H}\leq\\
\leq \limsup_{n\ra+\infty}\pa{f(x+h)-f(x+P(x,\alpha_n))+\gen{\nabla_H f_{\alpha_n}(x),P(x,\alpha_n)}_H}\leq\\
\leq f(x+h)-\liminf_{n\ra+\infty} f(x+P(x,\alpha))+\lim_{n\ra+\infty}\gen{\nabla_H f_{\alpha_n}(x),P(x,\alpha_n)}_H\leq f(x+h)-f(x)=F(h)-F(0)
\end{gather*}
and since $0\in \dom(F)$, then $y\in \partial F(0)$. By the fact that $h_0$ is an element of minimal norm in $\partial F(0)$, then all the inequalities in \eqref{dis lsc} are actually equalities. So
\begin{gather*}
\lim_{n\ra+\infty}\abs{\nabla_H f_{\alpha_n}(x)}_H=\abs{h_0}_H.
\end{gather*}
So by inequality \eqref{disuguaglianza 3.1} we have $\lim_{n\ra+\infty}\abs{\nabla_H f_{\alpha_n}(x)-h_0}_H=0$. Since $f$ belongs to $W^{1,p}(X,\mu)$ for some $p>1$, $f_{\alpha_n}$ converges to $f$ in $\elle^p(X,\mu)$ as $n$ goes to $+\infty$ (Proposition \ref{proprieta MY}\eqref{convergenza MY}) and $f_{\alpha}\in W^{2,p}(X,\mu)$ (Proposition \ref{proprieta MY}\eqref{Sobolev MY}), then $h_0=\nabla_H f(x)$ for $\mu$-a.e. $x\in X$. So we get inequality \eqref{monotonia convergenza gradiente MY} using inequality \eqref{disuguaglianza 3}.

We have proved that for every sequence $(\alpha_n)_{n\in\N}$ converging to zero, there exists a subsequence $(\alpha_{n_k})_{k\in\N}$ such that
\[\lim_{k\ra+\infty}\nabla_H f_{\alpha_{n_k}}=\nabla_H f\]
in $\elle^p(X,\mu;H)$. This implies that $\nabla_H f_\alpha$ converges to $\nabla_H f$ in $\elle^p(X,\mu;H)$ as $\alpha$ goes to zero. Monotonicity of the convergence of the norms follows from inequality \eqref{monoton MY}.
\end{proof}

\section{Sobolev regularity estimates}\label{Sobolev regularity estimates}

The purpose of this section is to prove Theorem \ref{Main theorem 1} and in order to do so we will use a penalization method similar to the one used in \cite{BDPT09}, \cite{BDPT11} and \cite{DPL15}. For $\alpha>0$ let $U_\alpha$ be the Moreau--Yosida approximation along $H$ of $U$, as defined in formula \eqref{M env}. We recall the following proposition (see \cite[Proposition 5.12]{CF16})
\begin{pro}\label{Integrabilita MY}
Assume Hypothesis \ref{ipotesi peso} holds and let $\alpha\in(0,1]$. Then $U_\alpha$ satisfies Hypothesis \ref{ipotesi peso} and $U_\alpha$ is differentiable along $H$ at every $x\in X$ with $\nabla_H U_\alpha$ $H$-Lipschitz. Moreover $e^{-U_\alpha}\in W^{1,p}(X,\mu)$, for every $p\geq 1$, and $U_\alpha\in W^{2,t}(X,\mu)$, where $t$ is given by Hypothesis \ref{ipotesi peso}.
\end{pro}

We approach the problem in $\Omega$ by penalized problems in the whole space $X$, replacing $U$ by
\begin{gather}\label{Problema penalizzato}
V_\alpha(x):=U_\alpha(x)+\frac{1}{2\alpha}d_H^2(x,\Omega).
\end{gather}
for $\alpha\in(0,1]$. Namely for $\alpha\in(0,1]$, we consider the problem
\begin{gather}\label{problema con alpha}
\lambda u_\alpha-L_{\nu_\alpha}u_\alpha=f
\end{gather}
where $\lambda>0$, $f\in\elle^2(X,\nu_\alpha)$, $\nu_\alpha=e^{-V_\alpha}\mu$ and $L_{\nu_\alpha}$ is the operator defined as
\begin{gather*}
D(L_{\nu_\alpha})=\bigg\{u\in W^{1,2}(X,\nu_\alpha)\,\Big |\, \text{there exists }v\in\elle^2(X,\nu_\alpha)\text{ such that }\\
 \int_X\gen{\nabla_H u,\nabla_H \varphi}d\nu_\alpha=-\int_Xv\varphi d\nu_\alpha\text{ for every }\varphi\in\fcon^\infty_b(X)\bigg\},
\end{gather*}
with $L_{\nu_\alpha}u=v$ if $u\in D(L_{\nu_\alpha})$.

We set \(D(\Omega)=\set{x\in X\tc H\cap(x-\Omega)\neq\emptyset}\), as in Section \ref{Projection on convex set along $H$}. We remark that $\Omega\subseteq D(\Omega)$ and if Hypothesis \ref{ipotesi dominio} holds, then $\mu(D(\Omega))=1$ (Lemma \ref{lemma translation invaritant}).

\begin{pro}\label{check proprieta V alpha}
Assume Hypotheses \ref{ipotesi dominio} and \ref{ipotesi peso} hold and let $\alpha\in(0,1]$. Then the following properties hold:
\begin{enumerate}
\item $V_\alpha$ is a convex and $H$-continuous function;

\item $V_\alpha$ is differentiable along $H$ at every point $x\in D(\Omega)$ with $\nabla_H V_\alpha$ $H$-Lipschitz;\label{Valpha gradiente Hlip}

\item $e^{-V_\alpha}\in W^{1,p}(X,\mu)$, for every $p\geq 1$;

\item $V_\alpha\in W^{2,t}(X,\mu)$, where $t$ is given by Hypothesis \ref{ipotesi peso};

\item $\lim_{\alpha\ra 0^+}V_\alpha(x)=\eqsys{U(x) & x\in\Omega;\\ +\infty & x\notin \Omega.}$\label{convergenza V alpha}
\end{enumerate}
\end{pro}

\begin{proof}
Proposition \ref{convessita e continuita distanza} says that $d^2(\cdot,\Omega)$ is convex and $H$-continuous, while from Proposition \ref{Integrabilita MY} we get convexity of $U_\alpha$. By Proposition \ref{proprieta MY}\eqref{differenziabilita MY} we get that the function $\Upsilon_x:H\ra \R\cup\set{+\infty}$ defined as $\Upsilon_x(h):=U_\alpha(x+h)$ is Fr\'echet differentiable for every $x\in\dom(U)$. By Hypothesis \ref{ipotesi peso} we have $\mu(\dom(U))=1$, so $U_\alpha$ is $H$-continuous. Therefore $V_\alpha$ is convex and $H$-continuous.

By Proposition \ref{punto minimo lip}, Proposition \ref{differenziabilita distanza} we get that $d^2_H(\cdot,\Omega)$ is differentiable along $H$ at every point $x\in D(\Omega)$ with $H$-Lipschitz gradient. Proposition \ref{Integrabilita MY} says that $U_\alpha$ is differentiable along $H$ at every point $x\in X$ with $H$-Lipschitz gradient. Then $V_\alpha$ is differentiable along $H$ at every point $x\in D(\Omega)$ with $H$-Lipschitz gradient.

Since
\begin{gather*}
\int_Xe^{-V_\alpha(x)}d\mu(x)\leq \int_Xe^{-U_\alpha(x)}d\mu(x),
\end{gather*}
for every $\alpha\in(0,1]$ and $\mu$-a.e. $x\in X$, applying Proposition \ref{Integrabilita MY} we get that $e^{-V_\alpha}\in \elle^p(X,\mu)$, for every $p\geq 1$. For every $x\in D(\Omega)$
\begin{gather}\label{gradiente e alla -Valpha}
\nabla_H e^{-V_\alpha(x)}=e^{-V_\alpha(x)}\pa{\nabla_H U_\alpha(x)+\frac{1}{2\alpha}\nabla_H(d^2_H(\cdot,\Omega)(x))}.
\end{gather}
By point \eqref{Valpha gradiente Hlip} the right side of equality \eqref{gradiente e alla -Valpha} is $H$-Lipschitz. By Theorem \ref{teorema 5.11.2} we get $e^{-V_\alpha}\in W^{1,p}(X,\mu)$, for every $p\geq 1$. Using the same argument we get $V_\alpha\in W^{2,t}(X,\mu)$, where $t$ is given by Hypothesis \ref{ipotesi peso}, for every $\alpha\in(0,1]$.

Finally equality \eqref{convergenza V alpha} follows from Proposition \ref{caratterizzation distanza zero} and Proposition \ref{proprieta MY}\eqref{convergenza MY}.
\end{proof}

By Proposition \ref{check proprieta V alpha} we can apply \cite[Theorem 5.10]{CF16} to problem \eqref{problema con alpha} and get the following maximal Sobolev regularity result.

\begin{thm}\label{Stime per lip}
Assume Hypothesis \ref{ipotesi peso} holds and let $\alpha\in(0,1]$, $\lambda>0$ and $f\in\elle^2(X,\nu_\alpha)$. Equation \eqref{problema con alpha} has a unique weak solution $u_\alpha$. Moreover $u_\alpha\in W^{2,2}(X,\nu_\alpha)$ and
\begin{gather}
\label{1 stime max per lip}\norm{u_\alpha}_{\elle^2(X,\nu_\alpha)}\leq\frac{1}{\lambda}\norm{f}_{\elle^2(X,\nu_\alpha)};\qquad \norm{\nabla_H u_\alpha}_{\elle^2(X,\nu_\alpha;H)}\leq\frac{1}{\sqrt{\lambda}}\norm{f}_{\elle^2(X,\nu_\alpha)};\\
\label{2 stime max per lip}\norm{\nabla_H^2 u_\alpha}_{\elle^2(X,\nu_\alpha;\mathcal{H}_2)}\leq \sqrt{2}\norm{f}_{\elle^2(X,\nu_\alpha)}.
\end{gather}
In addition, for every $\alpha\in(0,1]$, there exists a sequence $\{u_{\alpha}^{(n)}\}_{n\in\N}\subseteq \fcon^3_b(X)$ such that $u_{\alpha}^{(n)}$ converges to $u_\alpha$ in $W^{2,2}(X,\nu_\alpha)$ and $\lambda u_{\alpha}^{(n)}-L_{\nu_\alpha} u_{\alpha}^{(n)}$ converges to $f$ in $\elle^2(X,\nu_\alpha)$.
\end{thm}

Now we have all the ingredients necessary to prove Theorem \ref{Main theorem 1}. The proof is similar to the one in \cite[Section 3]{DPL15}.

\begin{proof}[Proof of Theorem \ref{Main theorem 1}]
Let $f\in\fcon^\infty_b(X)$. By Theorem \ref{Stime per lip} we get that, for every $\alpha\in(0,1]$, equation \eqref{problema con alpha} has a unique weak solution $u_\alpha\in W^{2,2}(X,\nu_\alpha)$ such that inequalities \eqref{1 stime max per lip} and inequality \eqref{2 stime max per lip} hold. Moreover for every $\varphi\in\fcon^\infty_b(X)$ we have
\begin{gather*}
\lambda \int_X u_\alpha\varphi d\nu_\alpha+\int_X\gen{\nabla_H u_\alpha,\nabla_H \varphi}_Hd\nu_\alpha=\int_Xf\varphi d\nu_\alpha.
\end{gather*}
By Proposition \ref{caratterizzation distanza zero} and Proposition \ref{proprieta MY}\eqref{convergenza MY} we have
\begin{gather}\label{disuguaglianza e alla -U e alla -Valpha}
e^{-U(x)}\leq e^{-U_\alpha(x)}= e^{-V_\alpha(x)}\qquad x\in \Omega.
\end{gather}
So we get the inclusion $W^{2,2}(\Omega,\nu_\alpha)\subseteq W^{2,2}(\Omega,\nu)$ for every $\alpha\in(0,1]$.

Let $\set{\alpha_n}_{n\in\N}$ be a sequence converging to zero such that $0< \alpha_n\leq 1$ for every $n\in\N$. By inequalities \eqref{1 stime max per lip} and inequality \eqref{2 stime max per lip} the set $\set{u_{\alpha_n}\tc n\in\N}$ is a bounded set in $W^{2,2}(\Omega,\nu)$. By weak compactness a subsequence, that we will still denote by $\set{\alpha_n}_{n\in\N}$, exists such that $u_{\alpha_n}$ weakly converges to an element $u\in W^{2,2}(\Omega,\nu)$. Without loss of generality we can assume that $u_{\alpha_n}$, $\nabla_H u_{\alpha_n}$ and $\nabla_H^2 u_{\alpha_n}$ converge pointwise $\mu$-a.e. respectively to $u$, $\nabla_H u$ and $\nabla_H^2 u$.

By inequalities \eqref{1 stime max per lip}, for every $n\in\N$, we have
\begin{gather}\label{stima per lebesgue in max reg 1}
\begin{array}{c}
\displaystyle \int_Xu_{\alpha_n}\varphi e^{-V_{\alpha_n}}d\mu\leq\pa{\int_X u_{\alpha_n}^2 e^{-V_{\alpha_n}}d\mu}^{\frac{1}{2}}\pa{\int_X \varphi^2 e^{-V_{\alpha_n}}d\mu}^{\frac{1}{2}}\leq\\
\displaystyle
\leq\frac{\norm{\varphi}_\infty}{\lambda}\pa{\int_X f^2 e^{-V_{\alpha_n}}d\mu}^{\frac{1}{2}}\pa{\int_X e^{-V_{\alpha_n}}d\mu}^{\frac{1}{2}}\leq \\
\displaystyle \leq \frac{\norm{f}_\infty\norm{\varphi}_\infty}{\lambda}\int_X e^{-V_{\alpha_n}}d\mu\leq\frac{\norm{f}_\infty\norm{\varphi}_\infty}{\lambda}\int_X e^{-V_1}d\mu.
\end{array}
\end{gather}
By inequality \eqref{stima per lebesgue in max reg 1}, Proposition \ref{check proprieta V alpha}\eqref{convergenza V alpha} and the Lebesgue dominated convergence theorem we get
\begin{gather}\label{convergenza con lebesgue in max reg 1}
\lim_{n\ra+\infty}\lambda\int_X u_{\alpha_n}\varphi d\nu_{\alpha_n}=\lambda\int_\Omega u\varphi d\nu.
\end{gather}

By inequalities \eqref{1 stime max per lip}, for every $n\in\N$, we have
\begin{gather}\label{stima per lebesgue in max reg 2}
\begin{array}{c}
\displaystyle \int_X\gen{\nabla_H u_{\alpha_n},\nabla_H \varphi}_H e^{-V_{\alpha_n}}d\mu\leq \int_X\abs{\nabla_H u_{\alpha_n}}_H\abs{\nabla_H \varphi}_H e^{-V_{\alpha_n}}d\mu\\
\displaystyle\leq \pa{\int_X\abs{\nabla_H u_{\alpha_n}}^2_H e^{-V_{\alpha_n}}d\mu}^{\frac{1}{2}}\pa{\int_X\abs{\nabla_H \varphi}^2_H e^{-V_{\alpha_n}}d\mu}^{\frac{1}{2}}\leq\\
\displaystyle\leq \frac{\norm{f}_\infty\norm{\abs{\nabla_H\varphi}_H}_\infty}{\sqrt{\lambda}}\int_X e^{-V_{\alpha_n}}d\mu\leq  \frac{\norm{f}_\infty\norm{\abs{\nabla_H\varphi}_H}_\infty}{\sqrt{\lambda}}\int_X e^{-V_{1}}d\mu.
\end{array}
\end{gather}
By inequality \eqref{stima per lebesgue in max reg 2}, Proposition \ref{check proprieta V alpha}\eqref{convergenza V alpha} and the Lebesgue dominated convergence theorem we get
\begin{gather}\label{convergenza con lebesgue in max reg 2}
\lim_{n\ra+\infty}\int_X\gen{\nabla_H u_{\alpha_n},\nabla_H \varphi}_H d\nu_{\alpha_n}=\int_\Omega\gen{\nabla_H u,\nabla_H \varphi}_H d\nu.
\end{gather}

Finally we have
\begin{gather}\label{stima per lebesgue in max reg 3}
\int_Xf\varphi e^{-V_{\alpha_n}}d\mu\leq \norm{f}_\infty\norm{\varphi}_\infty\int_Xe^{-V_{\alpha_n}}d\mu\leq \norm{f}_\infty\norm{\varphi}_\infty\int_Xe^{-V_1}d\mu
\end{gather}
By inequality \eqref{stima per lebesgue in max reg 3}, Proposition \ref{check proprieta V alpha}\eqref{convergenza V alpha} and the Lebesgue dominated convergence theorem we get
\begin{gather}\label{convergenza con lebesgue in max reg 3}
\lim_{n\ra+\infty}\int_Xf\varphi d\nu_{\alpha_n}=\int_\Omega f\varphi d\nu.
\end{gather}

Inequality \eqref{convergenza con lebesgue in max reg 1}, inequality \eqref{convergenza con lebesgue in max reg 2} and inequality \eqref{convergenza con lebesgue in max reg 3} give us that $u$ is a weak solution of equation \eqref{Problema}, i.e. for every $\varphi\in\fcon^\infty_b(X)$
\[\lambda\int_\Omega u\varphi d\nu+\int_\Omega\gen{\nabla_H u,\nabla_H \varphi}_H d\nu=\int_\Omega f\varphi d\nu.\]

By the lower semicontinuity of the norm of $\elle^2(\Omega,\mu)$ and $\elle^2(\Omega,\mu;H)$, inequalities \eqref{1 stime max per lip}, inequality \eqref{disuguaglianza e alla -U e alla -Valpha}, inequality \eqref{stima per lebesgue in max reg 3} and the Lebesgue dominated convergence theorem we get
\begin{gather*}
\norm{u}_{\elle^2(\Omega,\nu)}\leq\liminf_{n\ra+\infty}\norm{u_{\alpha_n}}_{\elle^2(\Omega,\nu)}\leq\liminf_{n\ra+\infty}\norm{u_{\alpha_n}}_{\elle^2(\Omega,{\nu_{\alpha_n}})}\leq\\
\leq \liminf_{n\ra+\infty}\norm{u_{\alpha_n}}_{\elle^2(X,\nu_{\alpha_n})}\leq\frac{1}{\lambda}\liminf_{n\ra+\infty}\norm{f}_{\elle^2(X,\nu_{\alpha_n})}=\frac{1}{\lambda}\norm{f}_{\elle^2(\Omega,\nu)};\\
\intertext{and}
\norm{\nabla_H u}_{\elle^2(\Omega,\nu;H)}\leq\liminf_{n\ra+\infty}\norm{\nabla_H u_{\alpha_n}}_{\elle^2(\Omega,\nu;H)}\leq\liminf_{n\ra+\infty}\norm{\nabla_H u_{\alpha_n}}_{\elle^2(\Omega,{\nu_{\alpha_n}};H)}\leq\\
\leq \liminf_{n\ra+\infty}\norm{\nabla_H u_{\alpha_n}}_{\elle^2(X,\nu_{\alpha_n};H)}\leq\frac{1}{\sqrt{\lambda}}\liminf_{n\ra+\infty}\norm{f}_{\elle^2(X,\nu_{\alpha_n};H)}=\frac{1}{\sqrt{\lambda}}\norm{f}_{\elle^2(\Omega,\nu;H)}.
\end{gather*}
In the same way by the lower semicontinuity of the norm of $\elle^2(\Omega,\mu;\mathcal{H}_2)$, inequality \eqref{2 stime max per lip}, inequality \eqref{disuguaglianza e alla -U e alla -Valpha}, inequality \eqref{stima per lebesgue in max reg 3} and the Lebesgue dominated convergence theorem we get
\begin{gather*}
\norm{\nabla^2_H u}_{\elle^2(\Omega,\nu;\mathcal{H}_2)}\leq\liminf_{n\ra+\infty}\norm{\nabla^2_H u_{\alpha_n}}_{\elle^2(\Omega,\nu;\mathcal{H}_2)}\leq\liminf_{n\ra+\infty}\norm{\nabla^2_H u_{\alpha_n}}_{\elle^2(\Omega,{\nu_{\alpha_n}};\mathcal{H}_2)}\leq\\
\leq \liminf_{n\ra+\infty}\norm{\nabla^2_H u_{\alpha_n}}_{\elle^2(X,\nu_{\alpha_n};\mathcal{H}_2)}\leq\sqrt{2}\liminf_{n\ra+\infty}\norm{f}_{\elle^2(X,\nu_{\alpha_n};\mathcal{H}_2)}=\sqrt{2}\norm{f}_{\elle^2(\Omega,\nu;\mathcal{H}_2)}.
\end{gather*}
If $f\in\elle^2(\Omega,\nu)$, a standard density argument gives us the assertions of our theorem.
\end{proof}

\section{The Neumann condition}\label{The Neumann condition}

We are now interested in proving Theorem \ref{inclusione dominio}. As in Section \ref{Sobolev regularity estimates} we approach the problem in $\Omega$ by penalized problems in the whole space $X$, replacing $U$ by the functions $V_\alpha$ defined via equation \eqref{Problema penalizzato}.

We start by proving a technical lemma that we will use in the proof of Theorem \ref{inclusione dominio}.

\begin{lemma}\label{lemma tecnico}
Assume Hypotheses \ref{ipotesi dominio} and \ref{ipotesi peso} hold and let $\alpha\in(0,1]$. Let $f\in\fcon^\infty_b(X)$ and let $u_\alpha$ be a weak solution of equation \eqref{problema con alpha}. For every $\varphi\in\fcon_b^\infty(X)$ the function
\[F_\alpha(x):=\varphi(x)\gen{\nabla_Hu_\alpha(x),\frac{\nabla_H G(x)}{\abs{\nabla_H G(x)}_H}}_He^{-V_\alpha(x)}\]
belongs to $W^{1,r}(\Omega,\mu)$ for every $1< r<2$. Furthermore we have $\trace F_\alpha\in \elle^q(G^{-1}(0),\rho)$ for every $1< q<2$ and
\[\lim_{\alpha\ra 0^+} \trace F_\alpha=\varphi\gen{\trace(\nabla_H u),\trace\pa{\frac{\nabla_H G}{\abs{\nabla_H G}_H}}}_H e^{-U},\]
where the limit is taken in $\elle^q(G^{-1}(0),\rho)$, for every $1< q<2$.
\end{lemma}

\begin{proof}
We start by proving that $F_\alpha\in \elle^r(\Omega,\mu)$ for every $1< r<2$.
\begin{gather*}
\int_\Omega\abs{F_\alpha}^rd\mu=\int_{\Omega}\abs{\varphi\gen{\nabla_Hu_\alpha,\nabla_H G}_H\frac{e^{-V_\alpha}}{\abs{\nabla_H G}_H}}^{r}d\mu\leq\norm{\varphi}_\infty\int_\Omega\abs{\nabla_H u_\alpha}_H^{r}e^{-rV_\alpha}d\mu=\\
=\norm{\varphi}_\infty\int_\Omega\abs{\nabla_H u_\alpha}_H^{r}e^{-(r-1)V_\alpha}e^{-V_\alpha}d\mu.
\end{gather*}
By using the H\"older inequality with $2/r$, for the measure $e^{-V_\alpha}\mu$, we get
\begin{gather}\label{boh 1}
\int_\Omega\abs{F_\alpha}^rd\mu\leq \norm{\varphi}_\infty\pa{\int_\Omega\abs{\nabla_H u_\alpha}_H^2 d\nu_\alpha}^{\frac{r}{2}}\pa{\int_\Omega e^{-\frac{r}{2-r}V_\alpha}d\mu}^{\frac{2-r}{2}}.
\end{gather}
By Proposition \ref{check proprieta V alpha}, Theorem \ref{Stime per lip} and the fact that $r/(2-r)>1$ we have
\begin{gather}\label{boh 2}
\begin{array}{c}
\displaystyle\int_\Omega\abs{F_\alpha}^rd\mu\leq \frac{\norm{\varphi}^r_\infty}{\lambda^{\frac{r}{2}}}\pa{\int_\Omega f^2e^{-V_\alpha} d\mu}^{\frac{r}{2}}\pa{\int_\Omega e^{-\frac{r}{2-r}V_\alpha}d\mu}^{\frac{2-r}{2}}\leq\\
\displaystyle\leq \frac{\norm{f}_\infty^r\norm{\varphi}^r_\infty}{\lambda^{\frac{r}{2}}}\pa{\int_\Omega e^{-V_\alpha} d\mu}^{\frac{r}{2}}\pa{\int_\Omega e^{-\frac{r}{2-r}V_\alpha}d\mu}^{\frac{2-r}{2}}\leq\\
\displaystyle\leq \frac{\norm{f}_\infty^r\norm{\varphi}^r_\infty}{\lambda^{\frac{r}{2}}}\pa{\int_X e^{-V_1} d\mu}^{\frac{r}{2}}\pa{\int_X e^{-\frac{r}{2-r}V_1}d\mu}^{\frac{2-r}{2}}.
\end{array}
\end{gather}
So $F_\alpha\in \elle^r(\Omega,\mu)$ for every $1< r<2$.

Now we want to prove that $\nabla_H F_\alpha\in\elle^r(\Omega,\mu;H)$, for every $1< r<2$. Observe that $\nabla_H F_\alpha$ exists by Hypotheses \ref{ipotesi dominio} and \ref{ipotesi peso} and
\begin{gather}\label{gradiente Falpha}
\begin{array}{c}
\displaystyle\nabla_H F_\alpha=\gen{\nabla_H u_\alpha,\nabla_H G}_H\frac{e^{-V_\alpha}}{\abs{\nabla_H G}_H}\nabla_H\varphi+\varphi\frac{e^{-V_\alpha}}{\abs{\nabla_H G}_H}\nabla_H^2 u_\alpha\nabla_H G+\\
\displaystyle+\varphi\frac{e^{-V_\alpha}}{\abs{\nabla_H G}_H}\nabla_H^2 G\nabla_H u_\alpha-\varphi\frac{\gen{\nabla_H u_\alpha,\nabla_H G}_H}{\abs{\nabla_H G}^3_H}e^{-V_\alpha}\nabla_H^2 G\nabla_H G+\\
\displaystyle-\varphi\frac{\gen{\nabla_H u_\alpha,\nabla_H G}_H}{\abs{\nabla_H G}_H}e^{-V_\alpha}\nabla_H V_\alpha.
\end{array}
\end{gather}
We will estimate each addend. We have
\begin{gather*}
\int_\Omega \abs{\gen{\nabla_H u_\alpha,\nabla_H G}_H\frac{e^{-V_\alpha}}{\abs{\nabla_H G}_H}\nabla_H\varphi}_H^rd\mu\leq \norm{\abs{\nabla_H\varphi}_H}_\infty^r\int_\Omega\abs{\nabla_H u_\alpha}_H^r e^{-rV_\alpha}d\mu.
\end{gather*}
Repeating the same arguments as in inequality \eqref{boh 1} and inequality \eqref{boh 2} we get
\begin{gather}\label{addendo 1}
\int_\Omega \abs{\gen{\nabla_H u_\alpha,\nabla_H G}_H\frac{e^{-V_\alpha}}{\abs{\nabla_H G}_H}\nabla_H\varphi}_H^rd\mu\leq \frac{\norm{f}_\infty^r\norm{\abs{\nabla_H\varphi}_H}_\infty^r}{\lambda^{\frac{r}{2}}}\pa{\int_X e^{-V_1} d\mu}^{\frac{r}{2}}\pa{\int_X e^{-\frac{r}{2-r}V_1}d\mu}^{\frac{2-r}{2}}.
\end{gather}
Recalling Theorem \ref{Stime per lip} and repeating some arguments used in inequality \eqref{boh 1} and in inequality \eqref{boh 2} we have
\begin{gather}\label{addendo 2}
\displaystyle\begin{array}{c}
\displaystyle\int_\Omega\abs{\varphi\frac{e^{-V_\alpha}}{\abs{\nabla_H G}_H}\nabla_H^2 u_\alpha\nabla_H G}_H^rd\mu\leq
\displaystyle\norm{\varphi}_\infty^r\int_\Omega\norm{\nabla_H^2 u_\alpha}_{\mathcal{H}_2}^re^{-rV_\alpha}d\mu\leq \\
\displaystyle\leq \norm{\varphi}^r_\infty\pa{\int_\Omega\norm{\nabla^2_H u_\alpha}_{\mathcal{H}_2}^2 d\nu_\alpha}^{\frac{r}{2}}\pa{\int_\Omega e^{-\frac{r}{2-r}V_\alpha}d\mu}^{\frac{2-r}{2}}\leq\\
\displaystyle\leq 2^{\frac{r}{2}} \norm{\varphi}^r_\infty\pa{\int_\Omega f^2 d\nu_\alpha}^{\frac{r}{2}}\pa{\int_\Omega e^{-\frac{r}{2-r}V_\alpha}d\mu}^{\frac{2-r}{2}}\leq\\
\displaystyle\leq 2^{\frac{r}{2}} \norm{\varphi}^r_\infty\norm{f}^r_\infty \pa{\int_\Omega e^{-V_\alpha} d\mu}^{\frac{r}{2}}\pa{\int_\Omega e^{-\frac{r}{2-r}V_\alpha}d\mu}^{\frac{2-r}{2}}\leq \\
\displaystyle\leq 2^{\frac{r}{2}} \norm{\varphi}^r_\infty\norm{f}^r_\infty \pa{\int_X e^{-V_1} d\mu}^{\frac{r}{2}}\pa{\int_X e^{-\frac{r}{2-r}V_1}d\mu}^{\frac{2-r}{2}}.
\end{array}
\end{gather}
Now we integrate the third addend of equality \eqref{gradiente Falpha},
\begin{gather*}
\int_\Omega \abs{\varphi\frac{e^{-V_\alpha}}{\abs{\nabla_H G}_H}\nabla_H^2 G\nabla_H u_\alpha}_H^rd\mu\leq  \norm{\varphi}_\infty^r\int_\Omega \pa{\frac{e^{-V_\alpha}}{\abs{\nabla_H G}_H}\norm{\nabla_H^2 G}_{\mathcal{H}_2}\abs{\nabla_H u_\alpha}_H}^rd\mu.
\end{gather*}
Applying H\"older inequality with an exponent $\beta>1$ such that $r\beta<2$ we get
\begin{gather*}
\int_\Omega \abs{\varphi\frac{e^{-V_\alpha}}{\abs{\nabla_H G}_H}\nabla_H^2 G\nabla_H u_\alpha}_H^rd\mu\leq \norm{\varphi}_\infty^r\pa{\int_\Omega\frac{\norm{\nabla_H^2 G}^{r\beta'}_{\mathcal{H}_2}}{\abs{\nabla_H G}^{r\beta'}_H}d\mu}^{\frac{1}{\beta'}}\pa{\int_\Omega\abs{\nabla_H u_\alpha}_H^{r\beta}e^{-r\beta V_\alpha}d\mu}^{\frac{1}{\beta}}.
\end{gather*}
By Proposition \ref{check proprieta V alpha}, Theorem \ref{Stime per lip} and the fact that $r\beta/(2-r\beta)>1$ we get
\begin{gather}\label{addendo 3}
\begin{array}{c}
\displaystyle\int_\Omega \abs{\varphi\frac{e^{-V_\alpha}}{\abs{\nabla_H G}_H}\nabla_H^2 G\nabla_H u_\alpha}_H^rd\mu\leq\\
\displaystyle\leq \norm{\varphi}_\infty^r\pa{\int_\Omega\frac{\norm{\nabla_H^2 G}^{r\beta'}_{\mathcal{H}_2}}{\abs{\nabla_H G}^{r\beta'}_H}d\mu}^{\frac{1}{\beta'}}\pa{\int_\Omega\abs{\nabla_H u_\alpha}_H^2d\nu_\alpha}^{\frac{r}{2}}\pa{\int_\Omega e^{-\frac{r\beta}{2-r\beta}V_\alpha}d\mu}^{\frac{2-r\beta}{2\beta}}\leq\\
\displaystyle\leq \frac{\norm{\varphi}_\infty^r}{\lambda^{\frac{r}{2}}}\pa{\int_\Omega\frac{\norm{\nabla_H^2 G}^{r\beta'}_{\mathcal{H}_2}}{\abs{\nabla_H G}^{r\beta'}_H}d\mu}^{\frac{1}{\beta'}} \pa{\int_\Omega f^2 e^{-V_\alpha}d\mu}^{\frac{r}{2}}\pa{\int_X e^{-\frac{r\beta}{2-r\beta}V_1}d\mu}^{\frac{2-r\beta}{2\beta}}\leq\\
\displaystyle\leq \frac{\norm{f}^r_\infty\norm{\varphi}_\infty^r}{\lambda^{\frac{r}{2}}}\pa{\int_\Omega\frac{\norm{\nabla_H^2 G}^{r\beta'}_{\mathcal{H}_2}}{\abs{\nabla_H G}^{r\beta'}_H}d\mu}^{\frac{1}{\beta'}} \pa{\int_X e^{-V_1}d\mu}^{\frac{r}{2}}\pa{\int_X e^{-\frac{r\beta}{2-r\beta}V_1}d\mu}^{\frac{2-r\beta}{2\beta}}.
\end{array}
\end{gather}
Arguing as in inequality \eqref{addendo 3}, then for the fourth addend of equality \eqref{gradiente Falpha} we have
\begin{gather}\label{addendo 4}
\begin{array}{c}
\displaystyle\int_\Omega\abs{\varphi\frac{\gen{\nabla_H u_\alpha,\nabla_H G}_H}{\abs{\nabla_H G}^3_H}e^{-V_\alpha}\nabla_H^2 G\nabla_H G}_H^rd\mu\leq  \norm{\varphi}_\infty^r\int_\Omega \pa{\frac{e^{-V_\alpha}}{\abs{\nabla_H G}_H}\norm{\nabla_H^2 G}_{\mathcal{H}_2}\abs{\nabla_H u_\alpha}_H}^rd\mu\leq\\
\displaystyle\leq\frac{\norm{f}^r_\infty\norm{\varphi}_\infty^r}{\lambda^{\frac{r}{2}}}\pa{\int_\Omega\frac{\norm{\nabla_H^2 G}^{r\beta'}_{\mathcal{H}_2}}{\abs{\nabla_H G}^{r\beta'}_H}d\mu}^{\frac{1}{\beta'}} \pa{\int_X e^{-V_1}d\mu}^{\frac{r}{2}}\pa{\int_X e^{-\frac{r\beta}{2-r\beta}V_1}d\mu}^{\frac{2-r\beta}{2\beta}}.
\end{array}
\end{gather}
Let $\beta>1$ such that $r\beta<2$. For the last addend of equality \eqref{gradiente Falpha} we obtain
\begin{gather*}
\int_\Omega\abs{\varphi\frac{\gen{\nabla_H u_\alpha,\nabla_H G}_H}{\abs{\nabla_H G}_H}e^{-V_\alpha}\nabla_H V_\alpha}_H^rd\mu\leq \norm{\varphi}_\infty^r\int_\Omega\pa{\abs{\nabla_H u_\alpha}_H e^{-V_\alpha}\abs{\nabla_H V_\alpha}_H}^rd\mu\leq\\
\leq \frac{\norm{f}_\infty^r\norm{\varphi}_\infty^r}{\lambda^{\frac{r}{2}}}\pa{\int_\Omega \abs{\nabla_H V_\alpha}_H^{r\beta'}d\mu}^{\frac{1}{\beta'}}\pa{\int_\Omega\pa{\abs{\nabla_H u_\alpha}_H e^{-V_\alpha}}^{r\beta}d\mu}^{\frac{1}{\beta}}.
\end{gather*}
Proceeding as in inequality \eqref{addendo 3} and recalling that $\nabla_H V_\alpha$ is $H$-Lipschitz (see Proposition \ref{check proprieta V alpha} and Theorem \ref{teorema 5.11.2}) we have
\begin{gather*}
\begin{array}{c}
\displaystyle\int_\Omega\abs{\varphi\frac{\gen{\nabla_H u_\alpha,\nabla_H G}_H}{\abs{\nabla_H G}_H}e^{-V_\alpha}\nabla_H V_\alpha}_H^rd\mu\leq\\
\displaystyle\leq \frac{\norm{f}_\infty^r\norm{\varphi}_\infty^r}{\lambda^{\frac{r}{2}}}\pa{\int_\Omega \abs{\nabla_H V_\alpha}_H^{r\beta'}d\mu}^{\frac{1}{\beta'}}\pa{\int_X e^{-V_1}d\mu}^{\frac{r}{2}}\pa{\int_X e^{-\frac{r\beta}{2-r\beta}V_1}d\mu}^{\frac{2-r\beta}{2\beta}}.
\end{array}
\end{gather*}
Finally recalling that $\abs{V_\alpha(x)}_H\leq \abs{U(x)}_H$ for every $x\in \Omega$ (see Proposition \ref{convergenza gradiente MY}) we get
\begin{gather}\label{addendo 5}
\begin{array}{c}
\displaystyle\int_\Omega\abs{\varphi\frac{\gen{\nabla_H u_\alpha,\nabla_H G}_H}{\abs{\nabla_H G}_H}e^{-V_\alpha}\nabla_H V_\alpha}_H^rd\mu\leq\\
\displaystyle\leq \frac{\norm{f}_\infty^r\norm{\varphi}_\infty^r}{\lambda^{\frac{r}{2}}}\pa{\int_\Omega \abs{\nabla_H U}_H^{r\beta'}d\mu}^{\frac{1}{\beta'}}\pa{\int_X e^{-V_1}d\mu}^{\frac{r}{2}}\pa{\int_X e^{-\frac{r\beta}{2-r\beta}V_1}d\mu}^{\frac{2-r\beta}{2\beta}}.
\end{array}
\end{gather}
By inequalities \eqref{addendo 1}, \eqref{addendo 2}, \eqref{addendo 3}, \eqref{addendo 4} and \eqref{addendo 5} we get that $F_\alpha$ belongs to $W^{1,r}(\Omega,\mu)$, for every $1< r<2$.

Observe that the final estimate of the inequalities \eqref{addendo 1}, \eqref{addendo 2}, \eqref{addendo 3}, \eqref{addendo 4} and \eqref{addendo 5} does not depend on $\alpha$. Then by Proposition \ref{convergenza gradiente MY}, Proposition \ref{check proprieta V alpha}, the Lebesgue dominated convergence theorem and Proposition \ref{trace continuity}, we get the furthermore part of our statement.
\end{proof}

We are now able to prove that if $u$ is a weak solution of problem \eqref{Problema}, then $u$ satisfies a Neumann type condition at the boundary.

\begin{proof}[Proof of Theorem \ref{inclusione dominio}]
By Theorem \ref{Main theorem 1} we get that for every $\varphi\in\fcon_b^\infty(X)$
\begin{gather}\label{weak solution}
\lambda\int_\Omega u\varphi d\nu+\int_\Omega\gen{\nabla_H\varphi,\nabla_H u}_Hd\nu=\int_\Omega f\varphi d\nu.
\end{gather}
Thanks to Proposition \ref{check proprieta V alpha} and Theorem \ref{Stime per lip}, equation \eqref{problema con alpha} has a unique solution $u_\alpha\in W^{2,2}(X,\nu_\alpha)$, for every $\alpha\in(0,1]$, such that inequalities \eqref{1 stime max per lip} and inequality \eqref{2 stime max per lip} hold. Moreover for every $\varphi\in\fcon^\infty_b(X)$ and $\alpha\in(0,1]$ we have
\begin{gather*}
\lambda \int_X u_\alpha\varphi d\nu_\alpha+\int_X\gen{\nabla_H u_\alpha,\nabla_H \varphi}_Hd\nu_\alpha=\int_Xf\varphi d\nu_\alpha.
\end{gather*}
In addition for every $\alpha\in(0,1]$ there exists a sequence $(u_{\alpha}^{(n)})_{n\in\N}\subseteq\fcon^3_b(X)$ such that
\begin{gather}
\label{convergenza W22 unk}W^{2,2}(X,\nu_\alpha)\text{-}\lim_{n\ra+\infty} u_\alpha^{(n)}=u_\alpha;\\
\label{convergenza L2 problema con k e n}\elle^2(X,\nu_\alpha)\text{-}\lim_{n\ra+\infty} \lambda u_{\alpha}^{(n)}-L_{\nu_\alpha} u_{\alpha}^{(n)}=f.
\end{gather}
Finally $u_\alpha$ converges to $u$ in $W^{2,2}(\Omega,\nu)$. For every $n\in\N$ and $\alpha\in(0,1]$ we set $f_\alpha^{(n)}:=\lambda u_{\alpha}^{(n)}-L_{\nu_\alpha} u_{\alpha}^{(n)}$, then the following equality holds
\begin{gather}\label{problema integrato}
\lambda\int_\Omega u_{\alpha}^{(n)}\varphi d\nu_\alpha-\int_\Omega \varphi L_{\nu_\alpha} u_\alpha^{(n)}d\nu_\alpha=\int_\Omega f_\alpha^{(n)}\varphi d\nu_\alpha
\end{gather}
for every $\varphi\in\fcon^\infty_b(X)$. By \cite[Proposition 5.3]{Fer15} we get that if $\psi\in\fcon^2_b(X)$ then
\begin{gather}\label{Lk:espressione esplicita}
L_{\nu_\alpha}\psi=\sum_{i=1}^{+\infty}\partial_{ii}\psi-\sum_{i=1}^{+\infty}(\partial_i V_\alpha+\hat{e}_i)\partial_i\psi
\end{gather}
where the series converges in $\elle^2(X,\nu_\alpha)$. Since $\elle^2(\Omega,\nu_\alpha)\subseteq\elle^2(\Omega,\nu)$, the series \eqref{Lk:espressione esplicita} also converges in $\elle^2(\Omega,\nu)$.
Using equality \eqref{Lk:espressione esplicita} and the integration by parts formula (Theorem \ref{divergence theorem with traces}) we get
\begin{gather*}
\int_\Omega \varphi L_{\nu_\alpha} u_\alpha^{(n)}d\nu_\alpha=\int_\Omega\varphi\pa{\sum_{i=1}^{+\infty}\partial_{ii}u_\alpha^{(n)}-(\partial_i V_\alpha+\hat{e}_i)\partial_iu_\alpha^{(n)}}d\nu_\alpha=\\
=\sum_{i=1}^{+\infty}\int_\Omega\varphi\pa{\partial_{ii}u_\alpha^{(n)}-(\partial_i V_\alpha+\hat{e}_i)\partial_iu_\alpha^{(n)}}d\nu_\alpha=\\
=\sum_{i=1}^{+\infty}\pa{-\int_\Omega\partial_i\varphi\partial_iu_\alpha^{(n)}d\nu_k+\int_{G^{-1}(0)}\varphi\trace(\partial_i u_\alpha^{(n)})\trace\pa{\frac{\partial_i G}{\abs{\nabla_H G}_H}}e^{-U_\alpha}d\rho}=\\
=-\int_\Omega\gen{\nabla_H\varphi,\nabla_H u_\alpha^{(n)}}_Hd\nu_\alpha+\int_{G^{-1}(0)}\varphi\gen{\trace(\nabla_Hu_\alpha^{(n)}),\trace\pa{\frac{\nabla_H G}{\abs{\nabla_H G}_H}}}_H e^{-U_\alpha}d\rho.
\end{gather*}
Arguing as in Lemma \ref{lemma tecnico} and recalling \eqref{convergenza W22 unk} we get
\[\lim_{n\ra+\infty}\int_{G^{-1}(0)}\varphi\gen{\trace(\nabla_Hu_\alpha^{(n)}),\trace\pa{\frac{\nabla_H G}{\abs{\nabla_H G}_H}}}_H e^{-U_\alpha}d\rho=\int_{G^{-1}(0)}\varphi\gen{\trace(\nabla_Hu_\alpha),\trace\pa{\frac{\nabla_H G}{\abs{\nabla_H G}_H}}}_H e^{-U_\alpha}d\rho.\]
By \eqref{convergenza L2 problema con k e n} and Proposition \ref{trace continuity}, letting $n\ra+\infty$ in equality \eqref{problema integrato} we get
\begin{gather}\label{problema integrato con alpha}
\lambda\int_\Omega u_{\alpha}\varphi d\nu_\alpha+\int_\Omega\gen{\nabla_H\varphi,\nabla_H u_\alpha}_Hd\nu_\alpha-\int_{G^{-1}(0)}\varphi\gen{\trace(\nabla_Hu_\alpha),\trace\pa{\frac{\nabla_H G}{\abs{\nabla_H G}_H}}}_H e^{-U_\alpha}d\rho=\int_\Omega f\varphi d\nu_\alpha.
\end{gather}
By Theorem \ref{Stime per lip} we get
\begin{gather*}
\begin{array}{c}
\displaystyle\int_\Omega u_\alpha\varphi e^{-V_\alpha}d\mu\leq \norm{\varphi}_\infty\pa{\int_\Omega u_\alpha^2e^{-V_\alpha}d\mu}^{\frac{1}{2}}\pa{\int_\Omega e^{-V_\alpha}d\mu}^{\frac{1}{2}}\leq \\
\displaystyle\leq \frac{\norm{\varphi}_\infty}{\lambda}\pa{\int_\Omega f^2 e^{-V_\alpha}d\mu}^{\frac{1}{2}}\pa{\int_\Omega e^{-V_1}d\mu}^{\frac{1}{2}}\leq \frac{\norm{f}_\infty\norm{\varphi}_\infty}{\lambda}\int_\Omega e^{-V_1}d\mu\leq \frac{\norm{f}_\infty\norm{\varphi}_\infty}{\lambda}\int_X e^{-V_1}d\mu,
\end{array}
\end{gather*}
and
\begin{gather*}
\begin{array}{c}
\displaystyle\int_\Omega \gen{\nabla_H u_\alpha,\nabla_H\varphi}_H e^{-V_\alpha}d\mu\leq \norm{\abs{\nabla_H\varphi}_H}_\infty\pa{\int_\Omega \abs{\nabla_H u_\alpha}_H^2e^{-V_\alpha}d\mu}^{\frac{1}{2}}\pa{\int_\Omega e^{-V_\alpha}d\mu}^{\frac{1}{2}}\leq \\
\displaystyle\leq \frac{\norm{\abs{\nabla_H\varphi}_H}_\infty}{\sqrt{\lambda}}\pa{\int_\Omega f^2 e^{-V_\alpha}d\mu}^{\frac{1}{2}}\pa{\int_\Omega e^{-V_1}d\mu}^{\frac{1}{2}}\\
\displaystyle\leq \frac{\norm{f}_\infty\norm{\abs{\nabla_H \varphi}_H}_\infty}{\sqrt{\lambda}}\int_\Omega e^{-V_1}d\mu\leq \frac{\norm{f}_\infty\norm{\abs{\nabla_H \varphi}_H}_\infty}{\sqrt{\lambda}}\int_X e^{-V_1}d\mu.
\end{array}
\end{gather*}
Moreover we have
\begin{gather*}
\int_\Omega f\varphi e^{-V_\alpha}d\mu\leq\norm{f}_\infty\norm{\varphi}_\infty\int_\Omega e^{-V_1}d\mu\leq\norm{f}_\infty\norm{\varphi}_\infty\int_X e^{-V_1}d\mu.
\end{gather*}
By Lemma \ref{lemma tecnico} the map
\[x\mapsto \varphi(x)\gen{\nabla_Hu_\alpha(x),\nabla_H G(x)}_H\frac{e^{-V_\alpha(x)}}{\abs{\nabla_H G(x)}_H}=:F_\alpha(x)\]
belongs to $W^{1,r}(\Omega,\mu)$, for every $1< r<2$. In particular $\trace F_\alpha\in\elle^q(G^{-1}(0),\rho)$ for every $1< q<2$.
Taking the limit $\alpha\ra 0^+$ in equality \eqref{problema integrato con alpha}, by Proposition \ref{check proprieta V alpha} and the Lebesgue dominated convergence theorem we get
\begin{gather}\label{limite fatto}
\lambda\int_\Omega u\varphi d\nu+\int_\Omega\gen{\nabla_H\varphi,\nabla_H u}_Hd\nu-\int_{G^{-1}(0)}\varphi\gen{\trace(\nabla_Hu),\trace\pa{\frac{\nabla_H G}{\abs{\nabla_H G}_H}}}_H e^{-U}d\rho=\int_\Omega f\varphi d\nu.
\end{gather}
Taking into consideration equality \eqref{weak solution}, then equality \eqref{limite fatto} becomes
\[\int_{G^{-1}(0)}\varphi\gen{\trace(\nabla_Hu),\trace\pa{\frac{\nabla_H G}{\abs{\nabla_H G}_H}}}_H e^{-U}d\rho=0,\]
for every $\varphi\in\fcon^\infty_b(X)$. By Proposition \ref{densita nel bordo} we get $\gen{\trace(\nabla_Hu)(x),\trace(\nabla_H G)(x)}_H=0$ for $\rho$-a.e. $x\in G^{-1}(0)$ .
\end{proof}

\section{Examples}\label{Examples}

In this section we show how our theory can be applied to some examples. Let $d\xi$ be the Lebesgue measure on $[0,1]$ and consider the classical Wiener measure $P^W$ on $\con[0,1]$ (see \cite[Example 2.3.11 and Remark 2.3.13]{Bog98} for its construction). The Cameron--Martin space $H$ is the space of the continuous functions $f$ on $[0,1]$ such that $f$ is absolutely continuous, $f'\in\elle^2([0,1],d\xi)$ and $f(0)=0$. Moreover $\abs{f}_H=\norm{f'}_{\elle^2([0,1],d\xi)}$ (see \cite[Lemma 2.3.14]{Bog98}). An orthonormal basis of $\elle^2([0,1],d\xi)$ is given by the functions
\[e_n(\xi)=\sqrt{2}\sin\frac{\xi}{\sqrt{\lambda_n}}\qquad\text{where }\lambda_n=\frac{4}{\pi^2(2 n-1)^2}\text{ for every }n\in\N.\]
We recall that if $f,g\in H$, then
\[\abs{f}^2_H=\sum_{i=1}^{+\infty}{\lambda_i}^{-1}\gen{f,e_i}^2_{\elle^2([0,1],d\xi)},\qquad \gen{f,g}_H=\sum_{i=1}^{+\infty}{\lambda_i}^{-1}\gen{f,e_i}_{\elle^2([0,1],d\xi)}\gen{g,e_i}_{\elle^2([0,1],d\xi)}.\]
Finally we remind the reader that an orthonormal basis for $H$ is given by the sequence $\{\sqrt{\lambda_k}e_k\,|\, k\in\N\}$.

\subsection{Admissible sets}

Let
\[G_{\sigma,c}(f)=\int_0^1f(\xi)d\sigma(\xi)-c,\qquad G^{(r)}(f)=\int_0^1\abs{f(\xi)}^2d\xi-r^2,\]
where $\sigma$ is a finite, non everywhere zero, Borel measure in $[0,1]$, $f\in\con[0,1]$ and $c,r\in\R$. Observe that the sets $G_{\sigma,c}^{-1}(-\infty,0]$ are halfspaces, since $G_{\sigma,c}\in (\con[0,1])^*$. Now we show that $G_{\sigma,c}$ and $G^{(r)}$ satisfy Hypothesis \ref{ipotesi dominio}.

Easy calculations give
\begin{gather}\label{gradienti G in esempi}
\nabla_H G_{\sigma,c}(f)=\sum_{i=1}^{+\infty} \sqrt{\lambda_i}\pa{\int_0^1 e_i(\xi)d\sigma(\xi)}(\sqrt{\lambda_i}e_i),\\
\nabla_H G^{(r)}(f)=2\sum_{i=1}^{+\infty}\sqrt{\lambda_i}\pa{\int_0^1f(\xi)e_i(\xi)d\xi}(\sqrt{\lambda_i}e_i).\label{gradienti G in esempi 2}
\end{gather}
So
\begin{gather*}
\abs{\nabla_H G_{\sigma,c}(f)}_H^2=\sum_{i=1}^{+\infty}\lambda_i\pa{\int_0^1 e_i(\xi)d\sigma(\xi)}^2,\\
\abs{\nabla_H G^{(r)}(f)}_H^2=4\sum_{i=1}^{+\infty}\lambda_i\pa{\int_0^1 f(\xi)e_i(\xi)d\xi}^2
\end{gather*}
Since $\sigma$ is non everywhere zero, then $\abs{\nabla_H G_{\sigma,c}(f)}_H$ is a non zero constant. So $\abs{\nabla_H G_{\sigma,c}}_H^{-1}$ belongs to every $\elle^q(\con[0,1],P^W)$ for every $q>1$. Now let $q>1$ and fix an integer $K$ bigger than a $q$, then
\begin{gather*}
\int_{\con[0,1]}\frac{1}{\abs{\nabla_H G^{(r)}(f)}_H^q}dP^W(f)=2^{-q}\int_{\con[0,1]}\pa{\sum_{i=1}^{+\infty}\lambda_i\pa{\int_0^1 f(\xi)e_i(\xi)d\xi}^2}^{-\frac{q}{2}}dP^W(f)\leq\\
\leq 2^{-q}\int_{\con[0,1]}\pa{\sum_{i=1}^{K}\lambda_i\pa{\int_0^1 f(\xi)e_i(\xi)d\xi}^2}^{-\frac{q}{2}}dP^W(f).
\end{gather*}
Since the maps $T:f\mapsto (\int_0^1 f(\xi)e_1(\xi)d\xi,\ldots,\int_0^1 f(\xi)e_K(\xi)d\xi)$ is linear and continuous, we can use the change of variable formula (see \cite[Formula (A.3.1)]{Bog98}) and obtain
\begin{gather}\label{integrazione di 1/}
\int_{\con[0,1]}\frac{1}{\abs{\nabla_H G^{(r)}(f)}_H^q}dP^W(f)\leq 2^{-q}\int_{\R^K}\pa{\sum_{i=1}^K\lambda_i\eta_i^2}^{-\frac{q}{2}}d P^W_K(\eta)\leq \pa{4\lambda_K}^{-\frac{q}{2}}\int_{\R^K}\norm{\eta}^{-q}dP^W_K(\eta),
\end{gather}
where $P^W_K$ the centered $K$-dimensional Gaussian measure given by $P^W_K:=P^W\circ T^{-1}$. The last integral in inequality \eqref{integrazione di 1/} is finite, since we took $K>q$. Thus both $G_{\sigma,c}$ and $G^{(r)}$ satisfy Hypothesis \ref{ipotesi dominio}\eqref{ipo dominio non degeneratezza}. Checking Hypothesis \ref{ipotesi dominio}\eqref{ipo dominio convessita e chiusura} is trivial.

Finally we have for $f\in\con[0,1]$
\begin{gather*}
\nabla_H^2 G_{\sigma,c}(f)=0,\qquad \nabla_H^2 G^{(r)}(f)=2\sum_{i=1}^{+\infty}\lambda_i\pa{(\sqrt{\lambda_i}e_i)\otimes(\sqrt{\lambda_i}e_i)}.
\end{gather*}
In particular $\|\nabla_H^2 G^{(r)}(f)\|^2_{\mathcal{H}_2}=\sum_{i=1}^{+\infty}\lambda_i^2=1/6$. Then $G_{\sigma,c}$ and $G^{(r)}$ satisfy all the conditions of Hypothesis \ref{ipotesi dominio}. We set $\Omega_{\sigma,c}:=G_{\sigma,c}^{-1}(-\infty,0]$ and $\Omega^{(r)}:=(G^{(r)})^{-1}(-\infty,0]$.

\subsection{An example of admissible weight (1)}

Let $\tau$ be a finite positive Borel measure in $[0,1]$. Consider the function $U:\con[0,1]\ra\R$ defined as
\[U(f)=\Phi\pa{\int_0^1 f(\xi)d\tau(\xi)},\]
where $\Phi:\R\ra\R$ is a $\con^1$ convex function such that for $\xi\in\R$
\begin{gather}\label{andamento della derivata di Phi}
\abs{\Phi'(\xi)}\leq C e^{\beta \abs{\xi}},
\end{gather}
for some $C\geq 0$ and $\beta>0$. Easy computations give that $U$ is a convex and continuous function. Using the fundamental theorem of calculus we get for every $\xi\in\R$
\begin{gather*}%\label{crescita Phi}
\abs{\Phi(\xi)}\leq \abs{\Phi(0)}+\frac{C}{\beta}e^{\beta\abs{\xi}}.
\end{gather*}
So
\begin{gather*}
\abs{U(f)}=\abs{\Phi\pa{\int_0^1f(\xi)d\tau(\xi)}}\leq\abs{\Phi(0)}+\frac{C}{\beta}e^{\beta\abs{\int_0^1 f(\xi)d\tau(\xi)}}\leq \abs{\Phi(0)}+\frac{C}{\beta}e^{\beta\norm{f}_\infty\norm{\tau}_{(\con[0,1])^*}}.
\end{gather*}
Therefore, by Fernique theorem, $U$ belongs to $\elle^t(\con[0,1],P^W)$ for every $t\geq 1$.

Observe that $U$ is Frech\'et differentiable with continuous derivative, since it is the composition of a element of $(\con[0,1])^*$ and a $\con^1(\R)$ function. By the chain rule for every $f,g\in\con[0,1]$ we have
\begin{gather*}
U'(f)(g)=\Phi'\pa{\int_0^1 f(\xi)d\tau(\xi)}\int_0^1g(\xi)d\tau(\xi).
\end{gather*}
So
\begin{gather*}
\abs{\nabla_H U(f)}_H^2=\sum_{n=1}^{+\infty}\abs{\partial_n U(f)}^2=\sum_{n=1}^{+\infty}\abs{U'(f)(\sqrt{\lambda_n}e_n)}^2=\\
=\pa{\Phi'\pa{\int_0^1 f(\xi)d\tau(\xi)}}^2\sum_{n=1}^{+\infty}\lambda_n\pa{\int_0^1e_n(\xi)d\tau(\xi)}^2\leq\\
\leq 2(\tau([0,1]))^2\pa{\Phi'\pa{\int_0^1 f(\xi)d\tau(\xi)}}^2\sum_{n=1}^{+\infty}\lambda_n=(\tau([0,1]))^2\pa{\Phi'\pa{\int_0^1 f(\xi)d\tau(\xi)}}^2.
\end{gather*}
By using inequality \eqref{andamento della derivata di Phi} we get
\begin{gather*}
\abs{\nabla_H U(f)}_H^2\leq C^2 (\tau([0,1]))^2 e^{2\beta\abs{\int_0^1 f(\xi)d\tau(\xi)}}\leq C^2 (\tau([0,1]))^2 e^{2\beta \norm{f}_\infty\norm{\tau}_{(\con[0,1])^*}}.
\end{gather*}
So, by Fernique's theorem, we get that $U$ belongs to $W^{1,t}(\con[0,1],P^W)$ for every $t\geq 1$. This implies that $U$ satisfies Hypothesis \ref{ipotesi peso}, since checking convexity and continuity of $U$ is trivial.

Consider the problem
\begin{gather}\label{Problema su semispazi}
\lambda u(f)-L_{e^{-U}P^W,\Omega_{\sigma,c}}u(f)=g(f),
\end{gather}
with data $\lambda>0$ and $g\in \elle^2(\Omega_{\sigma,c},e^{-U}P^W)$. By using Theorem \ref{Main theorem 1} we get that for every $\lambda>0$ and $g\in\elle^2(\Omega_{\sigma,c},e^{-U}P^W)$ problem \eqref{Problema su semispazi} has an unique weak solution $u\in W^{2,2}(\Omega_{\sigma,c},e^{-U}P^W)$. In addition the following inequalities hold
\begin{gather*}
\norm{u}_{\elle^2(\Omega_{\sigma,c},e^{-U}P^W)}\leq\frac{1}{\lambda}\norm{g}_{\elle^2(\Omega_{\sigma,c},e^{-U}P^W)};\qquad \norm{\nabla_H u}_{\elle^2(\Omega_{\sigma,c},e^{-U}P^W;H)}\leq\frac{1}{\sqrt{\lambda}}\norm{g}_{\elle^2(\Omega_{\sigma,c},e^{-U}P^W)};\\
\|\nabla_H^2 u\|_{\elle^2(\Omega_{\sigma,c},e^{-U}P^W;\mathcal{H}_2)}\leq \sqrt{2}\norm{g}_{\elle^2(\Omega_{\sigma,c},e^{-U}P^W)}.
\end{gather*}
Furthermore by Theorem \ref{inclusione dominio} we get that for $\rho$-a.e. $f\in G_{\sigma,c}^{-1}(0)$
\[\gen{\trace (\nabla_H u)(f),\trace (\nabla_H G_{\sigma,c})(f)}_H=0,\]
then by equality \eqref{gradienti G in esempi} we get for $\rho$-a.e. $f\in\con[0,1]$ with $\int_0^1f(\xi)d\sigma(\xi)=c$,
\[\sum_{i=1}^{+\infty}\sqrt{\lambda_i}(\trace \partial_i u(f))\pa{\int_0^1 e_i(\xi)d\sigma(\xi)}=0.\]

In a similar way by using Theorem \ref{Main theorem 1} to the problem
\begin{gather}\label{Problema su bolla}
\lambda u(f)-L_{e^{-U}P^W,\Omega^{(r)}}u(f)=g(f),
\end{gather}
with data $\lambda>0$ and $g\in \elle^2(\Omega^{(r)},e^{-U}P^W)$, we get that problem \eqref{Problema su bolla} has an unique weak solution $u\in W^{2,2}(\Omega^{(r)},e^{-U}P^W)$. In addition the following inequalities hold
\begin{gather*}
\norm{u}_{\elle^2(\Omega^{(r)},e^{-U}P^W)}\leq\frac{1}{\lambda}\norm{g}_{\elle^2(\Omega^{(r)},e^{-U}P^W)};\qquad \norm{\nabla_H u}_{\elle^2(\Omega^{(r)},e^{-U}P^W;H)}\leq\frac{1}{\sqrt{\lambda}}\norm{g}_{\elle^2(\Omega^{(r)},e^{-U}P^W)};\\
\|\nabla_H^2 u\|_{\elle^2(\Omega^{(r)},e^{-U}P^W;\mathcal{H}_2)}\leq \sqrt{2}\norm{g}_{\elle^2(\Omega^{(r)},e^{-U}P^W)}.
\end{gather*}
Moreover by Theorem \ref{inclusione dominio} we get that for $\rho$-a.e. $f\in (G^{(r)})^{-1}(0)$
\[\gen{\trace (\nabla_H u)(f),\trace (\nabla_H G^{(r)})(f)}_H=0,\]
then by equality \eqref{gradienti G in esempi 2} we get for $\rho$-a.e. $f\in\con[0,1]$ with $\norm{f}_2=r$,
\[\sum_{i=1}^{+\infty}\sqrt{\lambda_i}(\trace \partial_i u(f))\pa{\int_0^1 f(\xi)e_i(\xi)d\xi}=0.\]

\subsection{An example of admissible weight (2)}
Throughout this subsection we will assume that the following hypothesis holds.

\begin{hyp}\label{ipo esempio}
Let $\Psi \in \con^1(\R\times[0,1])$ be such that
\begin{enumerate}
\item for every fixed $r\in [0,1]$, the function $\Psi(\cdot,r)$ is convex;

\item\label{ipotesi_su_C(t)} for all $s\in\R$ and $r\in[0,1]$ we have
\[\abs{\frac{\partial \Psi}{\partial s}(s,r)}\leq C(r) e^{\beta \abs{s}}\]
where $\beta>0$ and $C(\cdot)$ is a non-negative function belonging to $\elle^2([0,1],d\xi)$;
\end{enumerate}
\end{hyp}
\noindent We want to show that the weight
\[U(f):=\int_0^1 \Psi(f(\xi),\xi)d\xi,\qquad f\in\con[0,1],\]
satisfies Hypothesis \ref{ipotesi peso}. First we remark that
\[\abs{\Psi(s,r)}\leq \abs{\Psi(0,r)}+C(r)\frac{e^{\beta \abs{s}}}{\beta}.\]
So for every $f\in\con[0,1]$ we get $\abs{U(f)}\leq\norm{\Psi(0,\cdot)}_\infty+\beta^{-1}\norm{C}_{\elle^2([0,1],d\xi)}e^{\beta\norm{f}_\infty}$, and by Fernique's theorem $U$ belongs to $\elle^t(\con[0,1],P^W)$ for every $t\geq 1$.

Observe that $U$ is a Frech\'et differentiable since it is the composition of a $\con^1$ function and a smooth function. In addition for every $f,g\in\con[0,1]$
\begin{gather*}
U'(f)(g)=\int_0^1\frac{\partial \Psi}{\partial s}(f(\xi),\xi)g(\xi)d\xi.
\end{gather*}
So we get
\begin{gather*}
\abs{\nabla_H U(f)}_H^2=\sum_{n=1}^{+\infty}\abs{\partial_n U(f)}^2=\sum_{n=1}^{+\infty}\abs{U'(f)(\sqrt{\lambda_n}e_n)}^2=\\
=\sum_{n=1}^{+\infty}\abs{\int_0^1\frac{\partial \Psi}{\partial s}(f(\xi),\xi)\sqrt{\lambda_n}e_n(\xi)d\xi}^2\leq 2\sum_{n=1}^{+\infty}\lambda_n\int_0^1\abs{\frac{\partial \Psi}{\partial s}(f(\xi),\xi)}^2d\xi.
\end{gather*}
Then by Hypothesis \ref{ipo esempio}\eqref{ipotesi_su_C(t)} we get
\begin{gather*}
\abs{\nabla_H U(f)}_H^2\leq 2\sum_{n=1}^{+\infty}\lambda_n\int_0^1C^2(\xi)e^{2\beta\abs{f(\xi)}}d\xi\leq e^{2\beta\norm{f}_\infty} \norm{C}_{\elle^2([0,1],d\xi)}^2.
\end{gather*}
Therefore, by Fernique's theorem, we get that $U$ belongs to $W^{1,t}(\con[0,1],P^W)$ for every $t\geq 1$.
So $U$ satisfies Hypothesis \ref{ipotesi peso}, since checking convexity and continuity is trivial.

Consider the problem
\begin{gather}\label{Problema su semispazi_bis}
\lambda u(f)-L_{e^{-U}P^W,\Omega_{\sigma,c}}u(f)=g(f),
\end{gather}
with data $\lambda>0$ and $g\in \elle^2(\Omega_{\sigma,c},e^{-U}P^W)$. By using Theorem \ref{Main theorem 1} we get that for every $\lambda>0$ and $g\in\elle^2(\Omega_{\sigma,c},e^{-U}P^W)$ problem \eqref{Problema su semispazi_bis} has an unique weak solution $u\in W^{2,2}(\Omega_{\sigma,c},e^{-U}P^W)$, and the following inequality holds
\begin{gather*}
\norm{u}_{W^{2,2}(\Omega_{\sigma,c},e^{-U}P^W)}\leq\pa{\frac{1}{\lambda}+\frac{1}{\sqrt{\lambda}}+\sqrt{2}}\norm{g}_{\elle^2(\Omega_{\sigma,c},e^{-U}P^W)}.
\end{gather*}
Furthermore by Theorem \ref{inclusione dominio} we get that for $\rho$-a.e. $f\in G_{\sigma,c}^{-1}(0)$
\[\gen{\trace (\nabla_H u)(f),\trace (\nabla_H G_{\sigma,c})(f)}_H=0,\]
then by equality \eqref{gradienti G in esempi} we get for $\rho$-a.e. $f\in\con[0,1]$ with $\int_0^1f(\xi)d\sigma(\xi)=c$,
\[\sum_{i=1}^{+\infty}\sqrt{\lambda_i}(\trace \partial_i u(f))\pa{\int_0^1 e_i(\xi)d\sigma(\xi)}=0.\]

In a similar way, by Theorem \ref{Main theorem 1}, the problem
\begin{gather}\label{Problema su bolla_bis}
\lambda u(f)-L_{e^{-U}P^W,\Omega^{(r)}}u(f)=g(f),
\end{gather}
has an unique weak solution $u\in W^{2,2}(\Omega^{(r)},e^{-U}P^W)$, whenever $\lambda>0$ and $g\in \elle^2(\Omega^{(r)},e^{-U}P^W)$. In addition
\begin{gather*}
\norm{u}_{W^{2,2}(\Omega_{\sigma,c},e^{-U}P^W)}\leq\pa{\frac{1}{\lambda}+\frac{1}{\sqrt{\lambda}}+\sqrt{2}}\norm{g}_{\elle^2(\Omega_{\sigma,c},e^{-U}P^W)}.
\end{gather*}
Moreover, by Theorem \ref{inclusione dominio}, if $u$ is the weak solution of \eqref{Problema su bolla_bis}, then for $\rho$-a.e. $f\in (G^{(r)})^{-1}(0)$
\[\gen{\trace (\nabla_H u)(f),\trace (\nabla_H G^{(r)})(f)}_H=0,\]
then by equality \eqref{gradienti G in esempi 2} we get for $\rho$-a.e. $f\in\con[0,1]$ with $\norm{f}_2=r$,
\[\sum_{i=1}^{+\infty}\sqrt{\lambda_i}(\trace \partial_i u(f))\pa{\int_0^1 f(\xi)e_i(\xi)d\xi}=0.\]

\begin{ack}
The authors would like to thank Alessandra Lunardi for many useful discussions and comments. This research was partially supported by the PRIN2010/11 grant ``Evolution differential problems: deterministic and stochastic approaches and their interactions''.
\end{ack}

\bibliographystyle{chicago}
\nocite{*} %\addcontentsline{toc}{section}{References}
\bibliography{bibpesoconvesso}
%\markboth{\textsc{References}}{\textsc{References}}

\end{document}